\documentclass[a4paper,fleqn,final]{cas-sc}

\ExplSyntaxOn
\cs_gset:Npn \__first_footerline:
  { \group_begin: \small \sffamily \__short_authors: \group_end: }
\ExplSyntaxOff

\usepackage[numbers]{natbib}

\def\tsc#1{\csdef{#1}{\textsc{\lowercase{#1}}\xspace}}
\tsc{WGM}
\tsc{QE}

\usepackage{amsmath,amssymb,mathtools,amsthm}
\usepackage{xcolor,graphicx,float,enumitem, marginnote}
\usepackage{comment}

\usepackage{bm}
\usepackage{placeins}

\usepackage{tikz-cd}
\usetikzlibrary{cd}

\usepackage{tikz}
\usetikzlibrary{matrix}
\tikzset
{
  over/.style={preaction={draw=white,-,line width=6pt}},
}

\graphicspath{{figures/}}

\newcommand{\ee}{\varepsilon}

\newcommand{\R}{{\mathbb R}}
\newcommand{\Rd}{{\R^d}}

\DeclareMathOperator{\diver}{div}

\DeclareMathOperator{\supp}{supp}

\newcommand*\diff{\mathop{}\!\mathrm{d}}

\newtheorem{theorem}{Theorem}[section]
\newtheorem{proposition}[theorem]{Proposition}
\newtheorem{corollary}[theorem]{Corollary}
\newtheorem{lemma}[theorem]{Lemma}

\theoremstyle{definition}
\newtheorem{definition}[theorem]{Definition}
\newtheorem{remark}[theorem]{Remark}
\newtheorem*{remark*}{Remark}
\newtheorem*{remarks*}{Remarks}

\numberwithin{equation}{section}

\newlist{enumeratetheorem}{enumerate}{2}
\setlist[enumeratetheorem,1]{
	label=\rm \roman*), 
	ref=Theorem \thetheorem--Item \roman*,
	topsep=0ex
	}
\setlist[enumeratetheorem,2]{label=\rm \alph*), ref=\theenumeratetheoremi.\alph*}

\newlist{enumeratedefinition}{enumerate}{2}
\setlist[enumeratedefinition,1]{
	label=\rm \roman*), 
	ref=Definition \thedefinition--Item \roman*,
	topsep=0ex}
\setlist[enumeratedefinition,2]{label=\rm \alph*), ref=\theenumeratedefinitioni.\alph*}

\newlist{enumeratesteps}{enumerate}{2}
\setlist[enumeratesteps,1]{
  label=\textit{Step \arabic*\protect\thisstep},
  ref=\arabic*,
  align=left, leftmargin=0pt, labelindent=!,
  listparindent=\parindent, labelwidth=0pt, itemindent=!
}
\setlist[enumeratesteps,2]{
    label= \textit{Step \theenumeratestepsi.\alph*\protect\thisstep},
    ref=\theenumeratestepsi.\alph*,
    align=left, leftmargin=0pt, labelindent=\parindent,
  listparindent=\parindent, labelwidth=0pt, itemindent=!
}
\newcommand{\step}[1][]{%
  \if\relax\detokenize{#1}\relax
    \def\thisstep{.}%
  \else
    \def\thisstep{: \protect #1.}%
  \fi
  \item}

\newcommand{\dx}{\, dx}
\newcommand{\dt}{\, dt}
\newcommand{\ds}{\, ds}

\newcommand{\dsigma}{\, d\sigma}

\newcommand{\mob}{\mathrm{m}}

\newcommand{\InvPhiPrimpp}{\varpi_{1}}
\newcommand{\InvPhiPrimnp}{\varpi_{2}}

\newcommand{\goodSemigroup}{free-energy dissipating semigroup}
\newcommand{\timeLimit}{time-limit operator}

\makeatletter
\renewcommand*{\@fnsymbol}[1]{\ensuremath{\ifcase#1\or \star \or \dagger\or \ddagger\or
		\mathsection\or \mathparagraph\or \|\or **\or \dagger\dagger
		\or \ddagger\ddagger \else\@ctrerr\fi}}
\makeatother

\usepackage{todonotes}
\ifdefined\DeclareUnicodeCharacter
  \DeclareUnicodeCharacter{2212}{-}
\fi

\ExplSyntaxOn
\ior_new:N \g__mainxviii_legacy_ior
\iow_new:N \g__mainxviii_legacy_iow
\bool_new:N \g__mainxviii_legacy_bool
\cs_new_protected:Npn \__mainxviii_clear_legacy:nn #1#2
  {
    \bool_gset_false:N \g__mainxviii_legacy_bool
    \file_if_exist:nT {#1}
      {
        \ior_open:Nn \g__mainxviii_legacy_ior {#1}
        \ior_str_map_inline:Nn \g__mainxviii_legacy_ior
          {
            \str_if_in:nnT {##1} {#2}
              {
                \bool_gset_true:N \g__mainxviii_legacy_bool
                \ior_map_break:
              }
          }
        \ior_close:N \g__mainxviii_legacy_ior
        \bool_if:NT \g__mainxviii_legacy_bool
          {
            \iow_open:Nn \g__mainxviii_legacy_iow {#1}
            \iow_now:Nn \g__mainxviii_legacy_iow {\relax}
            \iow_close:N \g__mainxviii_legacy_iow
            \iow_term:x {main18:~Cleared~obsolete~bibliography~data~from~#1.}
          }
      }
  }
\exp_args:Nx \__mainxviii_clear_legacy:nn {\c_sys_jobname_str.aux} {abx@aux@}
\exp_args:Nx \__mainxviii_clear_legacy:nn {\c_sys_jobname_str.bbl} {ver@biblatex.sty}
\ExplSyntaxOff

\begin{document}

\let\WriteBookmarks\relax
\def\floatpagepagefraction{1}
\def\textpagefraction{.001}

\shorttitle{Drift-diffusion equations with saturation on \texorpdfstring{$\mathbb{R}^d$}{R-d}}    

\shortauthors{Alejandro Fernández-Jiménez}  

\title [mode = title]{Drift-diffusion equations with saturation on \texorpdfstring{$\mathbb{R}^d$}{R-d}}  



%
 \author[1]{Alejandro Fernández-Jiménez}[type=editor,
      style=chinese,
      auid=000,
      bioid=1,
      orcid=0000-0002-4548-4374]




\ead{a.fernandezjimenez@vu.nl}



\affiliation[1]{organization={Amsterdam Center for Dynamics and Computation, Vrije Universiteit Amsterdam},
            addressline={De Boelelaan 1111}, 
            city={Amsterdam},
            postcode={1081 HV}, 
            state={North Holland},
            country={The Netherlands}}










\begin{abstract}
%
%
We study a class of nonlinear drift--diffusion equations on $\Rd$ with
saturation-type mobilities, not necessarily concave.
Under structural assumptions on the mobility and potentials,
bounded-domain approximations and local Caccioppoli estimates
yield weak solutions for all positive times and an
$L^1$-contraction estimate.
The same estimates provide local space--time compactness of
time translates, whose subsequential limits are weak solutions.
Under additional confinement assumptions, uniform mass-tail
estimates yield a mass-preserving, free-energy dissipating
$C_0$-semigroup of $L^1(\Rd)$-contractions.
In this confining setting, every trajectory converges strongly
in $L^1(\Rd)$ to a stationary state with the same mass, and the
resulting time-limit operator is an $L^1$-contraction.
Under the same assumptions, we also identify the unique global
and $L^1$-local minimiser of the free energy at each fixed mass
and prove that it is stationary under the semigroup.
\end{abstract}



\maketitle

\setcounter{tocdepth}{2}

\section{Introduction}

The aim of this manuscript is to study the equation
\begingroup
\renewcommand{\theHequation}{P}
\begin{equation}\label{eq:saturation}\tag{P}
    \frac{\partial \rho}{\partial t} = \diver \left( \mob (\rho ) \nabla (U'(\rho) + V ) \right) \quad \text{in } (0, \infty ) \times \Rd,
\end{equation}
\endgroup
with $U$ convex, $V$ a sufficiently regular potential and $\mob (\rho)$ a non-linear mobility of \textit{saturation}-type.
In the previous work \cite{Carrillo_FJ_Gomez-Castro25}, Carrillo,
Gómez-Castro, and the present author studied this problem on a bounded domain
with the natural no-flux condition. Here we consider the Cauchy problem on
$\Rd$.

Aggregation-diffusion equations with saturation-type mobility appear in mathematical biology to model chemotaxis with prevention of overcrowding \cite{Chalub_Rodrigues06, Hillen_Painter01, Carrillo_Murakawa_Sato_Togashi_Trush19}. This family of equations also arises in mathematical physics in descriptions of thin liquid films \cite{Lisini_Matthes_Savare12, Matthes_McCann_Savare09}, phase segregation \cite{Slepvcev08, Zamponi_Jungel17}, and the relaxation of a gas of fermions \cite{Kaniadakis95, Kaniadakis_Quarati93}, among other applications. In \cite{Carrillo_FJ_Gomez-Castro25}, the authors provide a complete theory for drift-diffusion equations with saturation-type mobility on bounded domains. Before that work, the literature concerning \eqref{eq:saturation} was comparatively limited and it focused mostly on the case of concave mobilities or dimension $1$. In \cite{Burger_DiFrancesco_Yasmin06, diFrancesco_Rosado08}, the authors study a Keller--Segel model with prevention of overcrowding for $\mob(\rho)=\rho(1-\rho)$. Subsequently, \cite{DiFrancesco_Fagioli_Radici19} considers a family of one-dimensional nonlocal interaction equations with saturation and derives a rigorous limit from discrete distributions. This result is extended in \cite{Fagioli_Radici18} to a family of one-dimensional aggregation-diffusion equations. Both works treat mobilities of the form $\mob(\rho)=\rho\phi(\rho)$, where $\phi$ is decreasing and positive only on a finite interval. In \cite{Fagioli_Tse22}, the authors establish well-posedness of entropy solutions for a family of one-dimensional nonlocal transport equations with a general saturation-type mobility.
More recently, \cite{Alamri25} studies the $C^\alpha$ regularity of the saturation problem \eqref{eq:saturation} when $\mob (\rho) = \rho \phi (\rho)$ for $\phi$ continuous. 

From the numerical point of view, \cite{BCH23} introduces an implicit finite-volume scheme for \eqref{eq:saturation} with an aggregation term and mobility $\mob(\rho)=\rho\phi(\rho)$, where $\phi$ is non-increasing and compactly supported. The convergence of this scheme is studied further in \cite{GomezCastro25}. For a more detailed introduction to saturation-type problems, we refer to \cite{BCH23,Carrillo_FJ_Gomez-Castro25}.

Problem \eqref{eq:saturation} can be seen as a formal Wasserstein-type gradient flow (see \cite{CarrilloLisiniSavareSlepcev2010, DolbeaultNazaretSavare2009, Santambrogio15, Carrillo_FJ_Gomez-Castro25}) in the sense 
\begin{equation*}
    \frac{\partial \rho}{\partial t} = \diver \left( \mob (\rho) \nabla \frac{\delta \mathcal{F}}{\delta \rho} [\rho] \right)
\end{equation*}
for the free-energy functional
\begin{equation}\label{eq:Energy}
    \mathcal{F}[\rho]
    = \int_{\Rd}U(\rho(x))\,\dx
      + \int_{\Rd}V(x)\rho(x)\,\dx .
\end{equation}
Wasserstein-type gradient-flow equations such as \eqref{eq:saturation} have been studied extensively. The case of linear mobility $\mob(\rho)=\rho$ is by now well understood in the context of aggregation-drift-diffusion equations; see, for instance, \cite{Carrillo_McCann_Villani03,Carrillo_Hittmeir_Volzone_Yao19, Carrillo_Esposito_Falco_FJ24, Elbar_FJ_Santambrogio25} and the reviews \cite{CCY19,GomezCastro23}. More recent results cover a wider class of mobilities, including power-law mobilities $\mob(\rho)=\rho^a$ with Newtonian interaction potentials \cite{Carrillo_Gomez-Castro_Vazquez22a,Carrillo_Gomez-Castro_Vazquez22b,decourcel_Elbar25}, and porous-medium equations with nonlocal pressure \cite{Stan_delTeso_Vazquez16,Stan_delTeso_Vazquez18,Stan_delTeso_Vazquez19}.

This manuscript focuses on a nonlinear mobility of \textit{saturation}-type, i.e., the support of the mobility is a finite interval. The main focus and contribution of the project is to understand how to expand the theory on \textit{saturation}-type problems to the case of $\Rd$. Since a gradient flow theory is still not available for the problem \eqref{eq:saturation} with non-concave mobilities, the difficulty lies in the fact that we do this purely with compactness arguments without relying on the compactness of the domain because we are taking it to be $\Rd$. 

\paragraph{Motivation and main contributions.}
The study of \eqref{eq:saturation} on $\Rd$ is natural both from the modelling and analytical points of view. In applications describing populations or particles subject to congestion, the bounded-domain formulation introduces an artificial impermeable boundary, whereas the Cauchy problem allows the confinement mechanism to be determined entirely by the drift potential $V$. The whole-space setting is also the natural framework for studying the competition between nonlinear diffusion, saturation and confinement, as well as the spatial localization and long-time concentration of solutions.

The passage from a bounded domain to $\Rd$ is not a purely technical extension. Compactness is no longer automatic, the free energy may fail to be finite unless its density is normalized at the vacuum, and both mass and energy may escape to infinity along approximating sequences. Moreover, the estimates available for the regularised problems are generally only local in space. The principal analytical difficulty is therefore to combine local Caccioppoli estimates with coercivity and mass-tail estimates, so that the compactness obtained on bounded sets can be upgraded to global $L^1(\Rd)$ convergence while preserving mass, free-energy dissipation and the semigroup structure. This strategy has been applied successfully for other cases, see for instance \cite{Bertozzi_Slepcev10,  Bedrossian_Rodriguez_Bertozzi11, Carrillo_FJ_Gomez-Castro24}. However, to our knowledge, this strategy has never been applied to the case of non-linear mobility.

Under suitable coercivity and confinement assumptions, we construct a free-energy dissipating $C_0$-semigroup of $L^1(\Rd)$-contractions for \eqref{eq:saturation}. 
Taking advantage of the bounded-domain contraction theory of \cite{Carrillo_FJ_Gomez-Castro25} we construct a solution on a bounded ball $B_R$. Afterwards, we take the limit in the radius $R$ through a combination of local Caccioppoli estimates with coercivity and mass-tail bounds. We also prove that every trajectory converges strongly in \(L^1(\mathbb R^d)\) to a stationary state with the same mass, and that the resulting time-limit operator is an \(L^1\)-contraction. Finally, we also identify the unique fixed-mass (global and local) minimiser of the free energy and prove that it is stationary under the semigroup. These results provide a whole-space framework for the further study of generalized Wasserstein formulations \cite{DolbeaultNazaretSavare2009, CarrilloLisiniSavareSlepcev2010}, minimizing-movement schemes \cite{Jordan_Kinderlehrer_Otto98} and uniqueness questions via entropy solutions \cite{Carrillo99} for saturation-type drift-diffusion equations.

\paragraph{Structure of the manuscript.} 
Section \ref{sec:Main results} introduces the assumptions and main results. In Section \ref{sec:Existence} we focus on existence of weak solutions and we prove Theorem \ref{thm:existence saturation}. In Section \ref{sec:omega limit} we study the $\omega$-limit of \eqref{eq:saturation} and we prove Theorem \ref{thm:Asymptotic (P0)}. Finally, in Section \ref{sec:Numerical experiments} we provide some numerical experiments relevant to the content of the manuscript.

\section{Assumptions and main results}\label{sec:Main results}
In this section we state the assumptions and the main results. The local
structural hypotheses are the following.

\begin{enumerate}[label=(H$_{\arabic*}$)]\label{hyp:H}
    \item\label{hyp:mobility} There is some $\alpha \in (0, \infty)$ such that $\mob(0) = \mob(\alpha) = 0$, and $\mob > 0$ in $(0,\alpha)$.
          We assume that
          $\mob \in W^{1,1}(0,\alpha)\cap C([0,\alpha])
          \cap C^1( 0,\alpha )$.
    \item\label{hyp:admissible data} We deal with initial data $
              \rho_0$ in the admissible class of densities $\mathcal A \coloneqq \{ \rho \in L^1(\Rd) : 0 \le \rho \le \alpha \}.$
    \item          \label{hyp:H_V}          $V$ is of class $C^2 (\Rd)$ and bounded from below.
          Without loss of generality we assume $V \geq 0$.
    \item          \label{hyp:H_U}          $U \in W^{1,1} ((0,\alpha)) \cap C^3 ((0, \alpha ))$, $U(0)=0$
          and convex $U'' \geq 0$.
          We assume that $U$ is not trivial, i.e., 
          \begin{equation}
              \label{eq:s0}
              \begin{aligned}
                   & \text{there exists } s_0 \in (0,\alpha ) \text{ such that } U'' (s_0) >0.
              \end{aligned}
          \end{equation}
    \item
        \label{hyp:Phi}
        We assume that the diffusion is continuous, in the sense that $\mob U''  \in L^1(0,\alpha)$  
        and we define
        \[
            \Phi(s) \coloneqq \int_{s_0}^s \mob(\tau) U''(\tau) \diff \tau.
        \]
        Furthermore, we also assume that $\Phi$ is strictly increasing at $0$ and $\alpha$, i.e., $\Phi(0) < \Phi (s) < \Phi (\alpha)$ for all $s \in (0, \alpha)$.
        Lastly, we impose a technical regularity condition which will be suitable for compactness estimates. For $s_0$ defined in \eqref{eq:s0} we assume
        \begin{align}
            \label{eq:hypothesis Phi/Phi' ee = 0}
            \sup_{s \in [s_0,\alpha] }\left|\frac{(\Phi (\alpha)-\Phi (s))\Phi(s)}{\Phi'(s)} \right| + \sup_{s \in [0,s_0] }
            \left|\frac{(\Phi (s)-\Phi (0))\Phi(s)}{\Phi'(s)} \right| & < \infty.
        \end{align}
        In addition, we assume that the one-sided limits
        \[
            \lim_{s\downarrow0}\Phi'(s), \qquad \lim_{s\uparrow\alpha}\Phi'(s), \qquad \text{exist in } [0,+\infty].
        \]
\end{enumerate}
Assumptions \ref{hyp:mobility}--\ref{hyp:Phi} are local in space. They
contain the hypotheses used for the bounded-domain theory in
\cite{Carrillo_FJ_Gomez-Castro25} and we will always take them into account throughout the manuscript. 
We have included an additional $W^{1,1}$-regularity in
assumption \ref{hyp:mobility}. This extra regularity is needed in order to construct a variation-preserving regularisation for an expanding-ball limit.   
\begin{remark}
    Let us notice that assumption \ref{hyp:mobility} is very broad and includes standard well-studied cases such as $\mob (s) = s (1-s)$ but also less studied cases, for instance when $\mob$ is not concave, for instance $\mob(s) = s^a (\alpha - s)^b$ for any $a,b > 0$. To our knowledge, this would be the first result for drift-diffusion equations on $\Rd$ with a general non-concave mobility. 
\end{remark}

\begin{remark}\label{rem:U strict convex}
Assumptions \hyperref[hyp:H]{(H)} already imply strict convexity for $U$, i.e.
\begingroup
\renewcommand{\theHequation}{SCU}
\begin{equation*}
                            U''(s)  >0 , \qquad \text{for } s \in (0,\alpha) .
\end{equation*}
\endgroup 
Indeed, $\Phi'=\mob U''\geq0$, $\Phi(s_0)=0$, and
$\Phi'(s_0)>0$ imply that
\[
    \Phi(s)<0 \quad\text{for }0<s<s_0,
    \qquad
    \Phi(s)>0 \quad\text{for }s_0<s<\alpha.
\]
Due to \ref{hyp:Phi}, for $0<s<s_0$, we have
$(\Phi(s)-\Phi(0))(-\Phi(s))>0$, whereas for
$s_0<s<\alpha$, we have
$(\Phi(\alpha)-\Phi(s))\Phi(s)>0$.
The quotient bounds in assumption \ref{hyp:Phi}
therefore imply that $\Phi'(s)>0$ for every
$s\in(0,\alpha)\setminus\{s_0\}$. 
Since $\Phi'(s_0)>0$ as well, and
$\mob>0$ in $(0,\alpha)$, we conclude that
\[
    U''(s)=\frac{\Phi'(s)}{\mob(s)}>0,
    \qquad s\in(0,\alpha).
\]
\end{remark}


\begin{remark}
    Notice that, since $\mob \in C^1((0,\alpha))$ and $U \in C^3((0,\alpha))$, then $\Phi \in C^2((0,\alpha))$. 
    These main hypothesis \hyperref[hyp:H]{(H)} 
    are satisfied for the Porous Medium / Fast Diffusion cases. In particular, we can consider $\Phi$ in $C^2((0,\alpha))$ and such that
    \[
        \Phi'(s) = \begin{dcases}
            a(s) s^{m_1-1}             & \text{if } s \in (0,s_1),       \\
            b(s) (\alpha - s)^{m_2-1} & \text{if } s \in (s_2, \alpha),
        \end{dcases}
    \]
    for $m_1, m_2 > 0$, $a \in C([0, s_1])$, $b \in C([s_2, \alpha])$ and $0 < \underline a \le a(s) \le \overline a, 0 < \underline b \le b(s) \le \overline b$.
    With respect to assumption \ref{hyp:Phi}, if $m_1, m_2 > 0$ then $\Phi$ fulfills \ref{hyp:Phi}. 
    Furthermore, we stress that analogously to \ref{hyp:mobility} and the previous remark, the mobility $\mob$ does not need concavity in order to fulfill \ref{hyp:Phi}. One example is the case $\mob (s) = s^2 (1 - s^2)$ and $U(s) = s^m/(m-1)$ with $m > 0$, or $U(s) = s \log s$ if $m=1$.
\end{remark}

The assumptions \hyperref[hyp:H]{(H)} yield local compactness on $\Rd$, but do not by themselves
exclude loss of mass.
To upgrade the construction from
local to global compactness and recover mass tightness, we impose the following two extra assumptions on the
free energy and the confining potential:
\begin{enumerate}[label=(MT$_{\arabic*}$),ref=(MT$_{\arabic*}$)]
    \item \label{hyp:free energy bdd below}
    There exist $\vartheta\in(0,1)$ and a nonnegative function
    $g\in L^1(\Rd)$ such that
    \begin{equation}\label{eq:negative internal energy confinement}
        [U(s)]_-
        \leq\vartheta V(x)s+g(x),
        \qquad \text{for a.e. }x\in\Rd\text{ and every }0\leq s\leq\alpha.
    \end{equation}
    Here $[a]_-\coloneqq\max\{-a,0\}$ denotes the negative part.
    \item \label{hyp:growth of V}
    The drift potential $V$ is coercive:
    \begin{equation}\label{eq:coercive potential}
        \inf_{\Rd\setminus B_\sigma}V(x)\longrightarrow+\infty
        \quad\text{as }\sigma\to+\infty.
    \end{equation}
\end{enumerate}
\begin{remark}
Assumption \ref{hyp:free energy bdd below}  ensures that the negative
internal energy cannot cancel the confining effect of $V$. Indeed,
\[
    \mathcal F[\rho]
    \geq (1-\vartheta)\int_{\Rd}V\rho\,\dx-\|g\|_{L^1(\Rd)}.
\]
Thus, bounded free energy controls the confinement moment.
Assumption \ref{hyp:growth of V} makes the confinement cost per unit
mass tend to infinity at large distances. Together, these assumptions
give uniform mass tightness on free-energy sublevels. 
The role of these assumptions is discussed further in
Subsection~\ref{sec:energy bdd below}, particularly in
Lemmas~\ref{lem:uniform tightness omega}
and~\ref{lem:tails degenerate ball semigroups}.
\end{remark}

After presenting the assumptions, we can recall the main results of the manuscript. In order to do that we first introduce the notion of weak solution and \goodSemigroup{}.

\begin{definition}[Weak solution]
    \label{def:Weak solutions P0}
    We say $\rho$ is a weak solution of the problem \eqref{eq:saturation} in
    $(0,T)\times\Rd$ if $\rho\in L^1((0,T)\times\Rd)$,
    $\Phi(\rho)\in L^2(0,T;H^1_{\mathrm{loc}}(\Rd))$, and
    \begin{equation*}
        \int_{\Rd}\rho_0\varphi(0)\,\dx
        +\int_0^T\int_{\Rd}\rho\,\partial_t\varphi\,\dx\,\dt
        =\int_0^T\int_{\Rd}
        \bigl(\nabla \Phi (\rho) + \mob(\rho)\nabla V\bigr)
        \cdot\nabla\varphi\,\dx\,\dt
    \end{equation*}
    for every $\varphi\in C_c^{\infty}([0,T)\times\Rd)$. Let us recall that $\nabla \Phi (\rho) = \mob (\rho) \nabla U'(\rho)$. 
\end{definition}

In addition to $\mathcal{A}$, we consider the following sets of initial data: 
\begin{subequations}
\begin{align*}
    \mathcal A_{\mathcal F}
    & \coloneqq
    \left\{\rho\in\mathcal A:\ U(\rho)\in L^1(\Rd)
    \text{ and }V\rho\in L^1(\Rd)\right\} 
\end{align*}
the class of admissible densities with finite free energy and
\begin{align*}
    \mathcal A_M & \coloneqq \{ \rho \in \mathcal{A} \, : \, \| \rho \|_{L^1(\Rd)} = M \}
\end{align*}
the set of admissible densities with fixed mass $M$.  
\end{subequations}
We set $\mathcal F[\rho]=+\infty$ on
$\mathcal A\setminus\mathcal A_{\mathcal F}$. Under
\ref{hyp:free energy bdd below}, this convention is equivalently
\[
 \mathcal F[\rho]
 =\int_{\Rd}\bigl(U(\rho)+V\rho+g\bigr)\,\dx-\int_{\Rd}g\,\dx.
\]
Indeed, the first integrand is nonnegative and is bounded below by
$(1-\vartheta)V\rho$. If its integral is finite, then
$V\rho\in L^1(\Rd)$; convexity and $U(0)=0$ give
$[U(s)]_+\leq C_Us$, and \eqref{eq:negative internal energy confinement}
then implies $U(\rho)\in L^1(\Rd)$.

We will work with the notion of semigroup of solutions (possibly nonunique) for the problem \eqref{eq:saturation} as follows.

\begin{definition}
    \label{def:Free-energy dissipating semigroup}
    We say that $S_t : \mathcal A \to \mathcal A$ is a \goodSemigroup{}  of solutions for \eqref{eq:saturation} if
    \begin{enumeratedefinition}
        \item
        \label{item:Semi-group weak solution}
        For $\rho_0 \in \mathcal A$, $\rho_t = S_{t}\rho_0$ is a weak solution to \eqref{eq:saturation}.
        \item
        \label{it:S^eps is C0}
        $S_t$ is a $C_0$-semigroup in $L^1$, i.e., for $t, h > 0$ we have $S_{t+h} = S_t S_h$ and $\lim_{t \to 0^+}\| S_t \rho_0 - \rho_0 \|_{L^1 (\Rd)} = 0$  for all $\rho_0 \in \mathcal A$.
        \item
        \label{it:S^eps is Lp continuous}
        $S_t : \mathcal A \to \mathcal A$ is an $L^1$-contraction,
        i.e., for any $\rho_0, \eta_0 \in \mathcal A$ then
        $
            \|S_t \rho_0 - S_t \eta_0 \|_{L^1(\Rd)} \le \|\rho_0 - \eta_0\|_{L^1(\Rd)}.
        $
        \item
        \label{item:H-1 equicontinuity}
        Free-energy dissipation and $C^{\frac 1 2}_{loc}([0,\infty); H^{-1}(\Rd))$ continuity:
        If $\rho_0 \in \mathcal A_{\mathcal F}$, then, writing
        $\rho_t=S_t\rho_0$, for all $0\leq t_1<t_2$ we have
        \begin{align}
            \label{eq:free energy dissipation}
            0 \leq \int_{t_1}^{t_2} \int_\Rd \mob(\rho_\sigma) |\nabla(U'(\rho_\sigma) + V)|^2 \dx \diff \sigma & \le \mathcal F[\rho_{t_1}] - \mathcal F[\rho_{t_2}], \\
            \label{eq:H-1 continuity}
            \| \rho_{t_2} -  \rho_{t_1} \|_{H^{-1}(\Rd)}      & \leq
            C \left( \mathcal{F} [\rho_{t_1}] - \mathcal{F} [\rho_{t_2}] \right)^{\frac{1}{2}} |t_2-t_1|^{\frac{1}{2}}.
        \end{align}
        In particular, $t \mapsto \mathcal F[S_t \rho_0]$ is non-increasing.
    \end{enumeratedefinition}
\end{definition}

With the notions of weak solution and \goodSemigroup{} we can introduce the first main result.
\begin{theorem}
    [Existence for \eqref{eq:saturation}]
    \label{thm:existence saturation}
    Assume \ref{hyp:mobility}--\ref{hyp:Phi}. It follows that:
    \begin{enumeratetheorem}
        \item\label{item:local existence theorem}
        There exists a family of maps $S_t\colon\mathcal A\to\mathcal A$
        such that, for every $\rho_0\in\mathcal A$ and $T>0$, the curve
        $\rho_t=S_t\rho_0$ is a weak solution of
        \eqref{eq:saturation} and
        \begin{equation*}
            \rho\in C([0,T];H^{-1}_{\mathrm{loc}}(\Rd)),
            \qquad
            \Phi(\rho)\in L^2(0,T;H^1_{\mathrm{loc}}(\Rd)),
            \qquad 0\leq\rho\leq\alpha.
        \end{equation*}
        The solution satisfy
        \begin{equation}\label{eq:main local mass inequality}
            \|S_t\rho_0\|_{L^1(\Rd)}
            \leq\|\rho_0\|_{L^1(\Rd)}
            \qquad\text{for every }t\geq0.
        \end{equation}
        Moreover, for every $t\geq0$, the maps $S_t$ satisfy
        \begin{equation}\label{eq:main local contraction}
            \|S_t\rho_0-S_t\eta_0\|_{L^1(\Rd)}
            \leq\|\rho_0-\eta_0\|_{L^1(\Rd)},
            \qquad \rho_0,\eta_0\in\mathcal A.
        \end{equation}

        \item\label{item:global existence theorem}
        If, in addition, we assume
        \ref{hyp:free energy bdd below}--\ref{hyp:growth of V}, then
        the construction is globally compact in space and no mass escapes
        at infinity. More precisely, the family of maps $S_t$ in the first part can be
        selected so that, for every $\rho_0\in\mathcal A$ 
        \begin{equation*}
            S_{\bullet}\rho_0\in C([0,T];L^1(\Rd))
            \cap C([0,T];H^{-1}(\Rd)),
        \end{equation*}
        and
        \begin{equation}\label{eq:main mass conservation}
            \|S_t\rho_0\|_{L^1(\Rd)}
            =\|\rho_0\|_{L^1(\Rd)},
            \qquad t\geq0.
        \end{equation}
        The family $S=(S_t)_{t\geq0}$ is a \goodSemigroup{} for
        \eqref{eq:saturation}. 
    \end{enumeratetheorem}
\end{theorem}

Concerning the minimisation of the free-energy functional, we can characterise the $L^1$-local minimisers by studying the corresponding Euler-Lagrange condition. In particular, the result (and proof) from \cite[Theorem 2.6]{Carrillo_FJ_Gomez-Castro25} also applies to this case. Therefore, we have the following remark.

\begin{remark}[Euler--Lagrange condition]
    \label{thm:Euler-Lagrange}
    Fix $M>0$. If $\widehat\rho\in\mathcal A_{\mathcal F}$ is an $L^1$-local minimiser of $\mathcal F$ on $\mathcal A_M \coloneqq \{ \rho \in \mathcal{A} \, : \, \| \rho \|_{L^1(\Rd)} = M \}$, then there exists $C \in \mathbb R$ such that
    \begin{equation*}
        \begin{aligned}
            U'(\widehat \rho(x)) + V(x) \ge C, & \qquad \text{ for a.e. } x \text{ such that } 0 \le \widehat \rho(x) <\alpha.  \\
            U'(\widehat \rho(x)) + V(x) \le C, & \qquad \text{ for a.e. } x \text{ such that } 0 < \widehat \rho(x) \le\alpha .
        \end{aligned}
    \end{equation*}
    Furthermore, if $U'$ is invertible, then
    \begin{equation}
        \label{eq:Euler-Lagrange for P}
        \widehat \rho (x)= \min \{[(U')^{-1} (C-V(x))]_+ , \alpha \}\quad \text{a.e.~in } \Rd.
    \end{equation}
    Lastly, under assumptions
    \ref{hyp:free energy bdd below}--\ref{hyp:growth of V}, there exists a unique $C$ such
    that \eqref{eq:Euler-Lagrange for P} has mass $M$. The existence and
    uniqueness of this constant on $\Rd$ are proved in the proof of
    \ref{item:minimiser free energy} below.

\end{remark}
Finally, we can give an interpretation of the long-time behaviour also in terms of semigroups. Let us introduce the notion of \timeLimit{}.
\begin{definition}\label{def:timeLimit}
    We say that a semigroup $S$ for a problem \eqref{eq:saturation}  has a \timeLimit{} $S_\infty : \mathcal A \to \mathcal A$ if all the following are satisfied:
    \begin{enumeratedefinition}
        \item For any $\rho_0 \in \mathcal A$ there exists a limit in time $S_t \rho_0 \to S_\infty \rho_0$ strongly in $L^1(\Rd)$ as $t \rightarrow \infty$. 
        \item $S_\infty$ is stationary for the semigroup, i.e., $S_t S_\infty = S_\infty$.
        \item For any $\rho_0 \in \mathcal A$, $S_\infty \rho_0$ is a constant-in-time weak solution to \eqref{eq:saturation}.
    \end{enumeratedefinition}
\end{definition}
Thus, we recover the following result about the $\omega$-limit of \eqref{eq:saturation}.

\begin{theorem}[Long-time behaviour for \eqref{eq:saturation}]
    \label{thm:Asymptotic (P0)}
    Assume \ref{hyp:mobility}--\ref{hyp:Phi}, and let $S$ be an evolution
    map constructed in \ref{item:local existence theorem}.
    \begin{enumeratetheorem}
        \item
        \label{item:local omega compactness}
        Let $\rho_0\in\mathcal A$ and $t_n\to+\infty$. There exist a
        subsequence, not relabelled, and a weak solution $\bar\rho$ of
        \eqref{eq:saturation} on $(0,+\infty)\times\Rd$ such that, for every
        $T,L>0$ and $1\leq p<+\infty$,
        \begin{align}
            S_{t_n+\,\bullet}\rho_0&\longrightarrow\bar\rho
            &&\text{in }C([0,T];H^{-1}(B_L)),
            \label{eq:theorem time translate local H-1}\\
            S_{t_n+\,\bullet}\rho_0&\longrightarrow\bar\rho
            &&\text{strongly in }L^p((0,T)\times B_L).
            \label{eq:theorem time translate local convergence}
        \end{align}
        Moreover, $0\leq\bar\rho\leq\alpha$ and
        \begin{equation}\label{eq:local omega mass inequality}
            \|\bar\rho_s\|_{L^1(\Rd)}
            \leq\|\rho_0\|_{L^1(\Rd)}
            \quad\text{for a.e. }s>0.
        \end{equation}
        This conclusion is local in space. For the remaining items, assume in addition \ref{hyp:free energy bdd below}--\ref{hyp:growth of V}, let $S$ be
        the \goodSemigroup{} constructed in
        \ref{item:global existence theorem}, and fix $M>0$. Then: 

        \item
        \label{item:Time limit operator}
        The convergence of the
        time translates is global in space, no mass is lost, and $S$ has a
        \timeLimit{}, denoted by $S_\infty$. Moreover, $S_\infty$ is an
        $L^1(\Rd)$-contraction.

        \item
        \label{item:Semi-group local minimiser}
        Every $L^1$-local minimiser
        $\widehat\rho\in\mathcal A_{\mathcal F}\cap\mathcal A_M$ of
        $\mathcal F$ over $\mathcal A_M$ is stationary for the semigroup. We have that
        \[
            S_t\widehat\rho=\widehat\rho
            \qquad\text{for every }t\geq0.
        \]

        \item
        \label{item:minimiser free energy}
        The density
        $\widehat\rho$ defined in \eqref{eq:Euler-Lagrange for P} is the
        unique $L^1$-local and global minimiser of the free energy
        \eqref{eq:Energy} over $\mathcal A_M$.
    \end{enumeratetheorem}
\end{theorem}

Finally, we want to remark that there are choices of the mobility $\mob$ and the potentials $U$ and $V$ for
which \eqref{eq:saturation} admits infinitely many stationary states in
addition to \eqref{eq:Euler-Lagrange for P}. This is discussed in \cite[Section 6.3]{Carrillo_FJ_Gomez-Castro25}. The example presented there for bounded domains also applies in the case on $\Rd$.

\section{Existence}\label{sec:Existence}

The existence strategy is as follows. First, taking into account the theory on bounded domains in \cite{Carrillo_FJ_Gomez-Castro25} we will consider a regularised problem over bounded balls $B_R$. We will study this problem and \textit{a priori} Caccioppoli local estimates. Afterwards, we will go from the regularised problem to \eqref{eq:saturation}, showing the stability of the \textit{a priori} Caccioppoli local estimates. We discuss this in Subsections \ref{sec:a priori} and \ref{sec:Limit ee to 0}. Afterwards, we take the limit over a sequence of bigger balls and under assumptions \ref{hyp:mobility}--\ref{hyp:Phi} we prove existence of local solutions over $\Rd$, Subsection \ref{sec:Existence local}. Finally, with the extra mass tightness assumptions \ref{hyp:free energy bdd below}--\ref{hyp:growth of V} we upgrade the notion of solution to global, Subsection \ref{sec:Existence global}.

In order to proceed with our construction, for $R>0$, we set
\begin{equation}\label{def:AR}
    \mathcal A_R\coloneqq
    \{\rho\in L^1(B_R):0\leq\rho\leq\alpha\}.
\end{equation}
For $\rho_0\in\mathcal A$ we write
$\rho_0^R=\rho_0\vert_{B_R}$. If $\rho\in\mathcal A_R$, its extension by
zero is denoted by $\rho^*=\rho\chi_{B_R}$.

\subsection{The regularised problem}\label{sec:a priori}

We begin with the uniformly parabolic approximation on $B_R$,
\begingroup
\renewcommand{\theHequation}{Pepsilon}
\begin{equation}\tag{P$_\ee$}\label{eq:saturation regularised}
    \partial_t\rho
    =\Delta\Phi_\ee(\rho)
      +\diver\bigl(\mob_\ee(\rho)\nabla V\bigr)
    \quad\text{in }(0,T)\times B_R,
\end{equation}
\endgroup
with the no-flux condition
\begin{equation}\label{eq:regularised no flux}
    \bigl(\nabla\Phi_\ee(\rho)
      +\mob_\ee(\rho)\nabla V\bigr)\cdot\nu=0
    \quad\text{on }(0,T)\times\partial B_R.
\end{equation}
We consider the regularisation introduced in
\cite[Section 2.2]{Carrillo_FJ_Gomez-Castro25}, which regularises both
the nonlinear diffusion and the mobility. For completeness, we describe
the construction and record the additional variation estimate required for
the expanding-ball limit. Given $\kappa(\ee)>0$ with
$\kappa(\ee)\to0$ as $\ee \to 0$, set
\[
 \underline\Phi_\ee'(s)
 \coloneqq\min\{\Phi'(s),\kappa(\ee)^{-1}\}+\ee.
\]
The diffusion is chosen so that $\Phi_\ee\in C^3([0,\alpha])$,
\begin{equation}\label{eq:regularised diffusion construction}
 \underline\Phi_\ee'(s)
 \leq\Phi_\ee'(s)
 \leq(1+\ee)\underline\Phi_\ee'(s),
 \qquad
 \Phi_\ee(s_0)=0,
 \qquad
 \Phi_\ee\longrightarrow\Phi
 \quad\text{in }C^2_{\mathrm{loc}}((0,\alpha)).
\end{equation}
Let us remark that \[\ee \leq \underline\Phi_\ee'(s) \leq \kappa (\ee)^{-1} + \ee. \] 
Thus, the problem has uniformly elliptic diffusion for every fixed $\ee > 0$. The mobility is chosen with
\begin{equation}\label{eq:regularised mobility construction}
 \begin{gathered}
  \mob_\ee\in C^1([0,\alpha]),\qquad
  \mob_\ee(0)=\mob_\ee(\alpha)=0,\qquad
  \mob_\ee(s)>0\quad(0<s<\alpha),\\
  |\mob_\ee'|>1\quad\text{in neighbourhoods of }0\text{ and }\alpha,\\
  (\ee,s)\longmapsto\mob_\ee(s)
  \in C([0,1]\times[0,\alpha])
  \cap C^\infty_{\mathrm{loc}}((0,1]\times[0,\alpha]),
  \qquad \mob_0=\mob.
 \end{gathered}
\end{equation}
Following \cite[Lemma 2.3]{Carrillo_FJ_Gomez-Castro25}, the truncation
scale and the mobility must be selected simultaneously. The next lemma
shows that the chosen approximations converge to the original setting and keep the relevant variational information. 

\begin{lemma}
    \label{lem:compatible regularisation}
    Under assumptions \ref{hyp:mobility}--\ref{hyp:Phi}, there exist a
    function $\kappa:(0,1]\to(0,\alpha/4)$ with
    $\kappa(\ee)\to0$ and families $\Phi_\ee$, $\mob_\ee$ satisfying
    \eqref{eq:regularised diffusion construction} and
    \eqref{eq:regularised mobility construction} such that
    \begin{equation}\label{eq:regularisation convergence}
    \begin{gathered}
        \Phi_\ee'\longrightarrow\Phi'
        \quad\text{in }L^1(0,\alpha),
        \qquad
        \Phi_\ee\longrightarrow\Phi
        \quad\text{in }C([0,\alpha]),\\
        \mob_\ee\longrightarrow\mob
        \quad\text{in }C([0,\alpha]),
        \qquad
        \mob_\ee\longrightarrow\mob
        \quad\text{in }C^1_{\mathrm{loc}}((0,\alpha)),
    \end{gathered}
    \end{equation}
    and
    \begin{equation}\label{eq:uniform regularised mobility variation}
       \sup_{0<\ee\leq1}
       \|\mob_\ee'\|_{L^1(0,\alpha)}
       =\sup_{0<\ee\leq1}
       \|\mob_\ee\|_{\mathrm{TV}(0,\alpha)}<+\infty.
    \end{equation}
\end{lemma}

The regularisation construction is provided by \cite[Lemma 2.3]{Carrillo_FJ_Gomez-Castro25}. We give only the additional argument showing that, under the assumption $\mob\in W^{1,1}(0,\alpha)$, the approximating mobilities can be chosen to satisfy $\sup_{\varepsilon\in(0,1)}\|\mob_\varepsilon'\|_{L^1(0,\alpha)}<\infty$. For the remaining parts of the construction and proof, we refer to the cited lemma.

\begin{proof}
Choose the scale $\kappa=\kappa(\ee)$ as in
\cite[Lemma 2.3]{Carrillo_FJ_Gomez-Castro25}.
We describe only the additional argument needed to obtain the
uniform total-variation bound.

Choose $c_{\ee,0}\to0$ such that
\[
 c_{\ee,0}>
 \max\left\{\kappa,\sup_{0\leq s\leq2\kappa}\mob(s)\right\}.
\]
On $[0,\kappa]$, set
$\widetilde\mob_\ee(s)=c_{\ee,0}s/\kappa$.
On $[\kappa,2\kappa]$, choose a nonincreasing majorant of $\mob$
joining $c_{\ee,0}$ to $\mob(2\kappa)$.
Use the reflected construction near $\alpha$, with height
$c_{\ee,\alpha}\to0$, and leave $\mob$ unchanged on
$[2\kappa,\alpha-2\kappa]$.
Since $\mob\in W^{1,1}(0,\alpha)$, the resulting preliminary
mobility satisfies
\[
 \|\widetilde\mob_\ee\|_{\mathrm{TV}(0,\alpha)}
 \leq \|\mob\|_{\mathrm{TV}(0,\alpha)}
       +2c_{\ee,0}+2c_{\ee,\alpha}.
\]
We next smooth $\widetilde\mob_\ee$ from above, retaining its
linear endpoint profiles, with a uniformly vanishing approximation
error and a vanishing increase in total variation:
\[
 \|\mob_\ee\|_{\mathrm{TV}(0,\alpha)}
 \leq \|\widetilde\mob_\ee\|_{\mathrm{TV}(0,\alpha)}+o(1).
\]
The smoothing is also chosen to give local $C^1$ convergence.
Smooth dependence on $\ee$ is obtained by local smoothing and
convex interpolation in the parameter, preserving the lower bounds
\[
 \mob_\ee\geq\mob
 \quad\text{on }[\kappa,\alpha-\kappa],
 \qquad
 \mob_\ee(s)\geq s,\quad
 \mob_\ee(\alpha-s)\geq s
 \quad\text{for }0<s<\kappa.
\]
Since the interpolation weights depend only on $\ee$, convex
interpolation preserves the uniform total-variation bound.

Uniform convergence to $\mob$ follows from the vanishing heights
of the endpoint modifications and the uniformly vanishing smoothing
errors. The remaining parts of the construction and proof are as in
\cite[Lemma 2.3]{Carrillo_FJ_Gomez-Castro25}.
\end{proof}

Let us also pay some extra attention to the natural regularisation $U_\ee$ of the diffusion potential $U$. 
\begin{lemma}
    \label{lem:convergence regularised internal energies}
Assume \ref{hyp:mobility}--\ref{hyp:Phi}. Let $\Phi_\ee$ and
$\mob_\ee$ be the approximations constructed in
Lemma~\ref{lem:compatible regularisation}. We can define a regularised diffusion potential $U_\ee$ given by
    \begin{equation*}
     U_\ee''(s)=\frac{\Phi_\ee'(s)}{\mob_\ee(s)},
     \qquad
     U_\ee\left(\frac\alpha2\right)=U\left(\frac\alpha2\right),
     \qquad
     U_\ee'\left(\frac\alpha2\right)=U'\left(\frac\alpha2\right) .
    \end{equation*}
Then the continuous extensions of $U_\ee$ to $[0,\alpha]$ satisfy 
\begin{equation}
\label{eq:convergence regularised internal energies}
\begin{gathered}
 U_\ee\longrightarrow U
 \quad\text{in }C([0,\alpha]),
 \qquad
 U_\ee'\longrightarrow U'
 \quad\text{in }L^1(0,\alpha),\\
 U_\ee'\longrightarrow U',
 \qquad
 U_\ee''\longrightarrow U''
 \quad\text{locally uniformly in }(0,\alpha).
\end{gathered}
\end{equation}
\end{lemma}

\begin{proof}
On every compact subinterval of $(0,\alpha)$, the convergence
of $\Phi_\ee'$ and $\mob_\ee$, together with the positivity
of $\mob$, gives
\[
 U_\ee''=\frac{\Phi_\ee'}{\mob_\ee}
 \longrightarrow \frac{\Phi'}{\mob}=U''
 \quad\text{locally uniformly in }(0,\alpha).
\]
To obtain convergence up to the endpoints, we establish
convergence of the second derivatives in a weighted
$L^1$ space. Set
\[
 \ell(s)=\min\{s,\alpha-s\}.
\]
By convexity and Tonelli's theorem,
\[
\begin{aligned}
    Q
    \coloneqq & \int_0^\alpha \ell(s)U''(s)\,\diff s
      =\int_0^{\alpha/2}sU''(s)\,\diff s
       +\int_{\alpha/2}^\alpha(\alpha-s)U''(s)\,\diff s\\
    = & \int_0^{\alpha/2}
        \left(\int_s^{\alpha/2}U''(r)\,\diff r\right)\diff s
      +\int_{\alpha/2}^\alpha
        \left(\int_{\alpha/2}^sU''(r)\,\diff r\right)\diff s\\
    = & \int_0^\alpha
        |U'(s)-U'(\alpha/2)|\,\diff s<\infty,
\end{aligned}
\]
where finiteness follows from $U'\in L^1(0,\alpha)$
and $U'(\alpha/2)\in\mathbb R$. 
Let us now consider
\[
 I_\ee=(\kappa(\ee),\alpha-\kappa(\ee)),
 \qquad
 E_\ee=(0,\alpha)\setminus I_\ee,
 \qquad
 J_\ee=\int_{I_\ee}\frac{\ell(s)}{\mob(s)}\,\diff s.
\]
We take the choice of $\kappa$ from 
Lemma \ref{lem:compatible regularisation} which ensures that
$\ee J_\ee\leq\sqrt{\ee}$ for $\ee$ small enough.
In particular, $\ee J_\ee\to0$. 
On $I_\ee$, the lower bound $\mob_\ee\geq\mob$ and
the estimate $\Phi_\ee'\leq(1+\ee)(\Phi'+\ee)$ imply
\[
 \int_{I_\ee}\ell(s)U_\ee''(s)\,\diff s
 \leq(1+\ee)\left(
       \int_{I_\ee}\ell(s)U''(s)\,\diff s
       +\ee J_\ee\right).
\]
On the set $E_\ee$, the construction instead gives
$\mob_\ee\geq\ell$. Consequently, it follows that
\[
 \begin{aligned}
 \int_{E_\ee}\ell(s)U_\ee''(s)\,\diff s
 &\leq(1+\ee)\int_{E_\ee}(\Phi'(s)+\ee)\,\diff s =(1+\ee)\left(
       \int_{E_\ee}\Phi'(s)\,\diff s
       +2\ee\kappa(\ee)\right)
       \longrightarrow0,
 \end{aligned}
\]
since $\Phi'\in L^1(0,\alpha)$ and $\kappa(\ee)\to0$.
Thus, the modifications near the endpoints make a vanishing
contribution to the weighted integrals of $U_\ee''$. 
Adding these estimates together yields that
\[
 \limsup_{\ee\to0}Q_\ee\leq Q,
 \qquad \text{where }
 Q_\ee:=\int_0^\alpha\ell(s)U_\ee''(s)\,\diff s.
\]
On the other hand, $U_\ee''\geq0$ and
$U_\ee''\to U''$ pointwise in $(0,\alpha)$. Therefore, Fatou's
lemma gives $Q\leq\liminf_{\ee\to0}Q_\ee$.
Hence, $Q_\ee\to Q$. 
Moreover, the Dominated Convergence Theorem applies to
\[\ell\min\{U_\ee'',U''\},\] which is bounded above by
the integrable function \[\ell U''.\] Therefore,
\[
 \begin{aligned}
 \int_0^\alpha
       \ell(s)|U_\ee''(s)-U''(s)|\,\diff s
 &=Q_\ee+Q
   -2\int_0^\alpha
       \ell(s)\min\{U_\ee''(s),U''(s)\}\,\diff s \longrightarrow 0.
 \end{aligned}
\]
We consider the normalisation at $\alpha/2$ given by $U_\ee \left( \frac{\alpha}{2} \right) = U \left( \frac{\alpha}{2} \right) $,  $U_\ee' \left( \frac{\alpha}{2} \right) = U' \left( \frac{\alpha}{2} \right) $ to integrate
this convergence. Since
\[
 U_\ee'(s)-U'(s)
 =\int_{\alpha/2}^{s}
       \bigl(U_\ee''(r)-U''(r)\bigr)\,\diff r,
\]
Fubini's theorem gives
\[
 \|U_\ee'-U'\|_{L^1(0,\alpha)}
 \leq\int_0^\alpha
       \ell(s)|U_\ee''(s)-U''(s)|\,\diff s
 \longrightarrow0.
\]
By \ref{hyp:H_U}, $U'\in L^1(0,\alpha)$, and the preceding
estimate also gives $U_\ee'\in L^1(0,\alpha)$.
Hence both functions extend continuously to $[0,\alpha]$.
Using $U_\ee(\alpha/2)=U(\alpha/2)$ and the fundamental
theorem of calculus, we obtain
\[
 \|U_\ee-U\|_{C([0,\alpha])}
 \leq\|U_\ee'-U'\|_{L^1(0,\alpha)}
 \longrightarrow0.
\]
Finally, integrating the local uniform convergence of
$U_\ee''$ from $\alpha/2$ gives local uniform convergence
of $U_\ee'$ to $U'$. This proves all the stated
convergences.
\end{proof}

We point out that the construction preserves the quantitative endpoint control in
assumption \ref{hyp:Phi}:
\begin{equation}\label{eq:regularised ratio bound}
 \sup_{0<\ee\leq1}
 \left( \sup_{s_0\leq s\leq\alpha}
 \left|\frac{[\Phi_\ee(\alpha)-\Phi_\ee(s)]\Phi_\ee(s)}
 {\Phi_\ee'(s)}\right|
 + \sup_{0\leq s\leq s_0}
 \left|\frac{[\Phi_\ee(s)-\Phi_\ee(0)]\Phi_\ee(s)}
 {\Phi_\ee'(s)}\right|
 \right)<+\infty.
\end{equation}
This is \cite[Lemma 3.2]{Carrillo_FJ_Gomez-Castro25} for the
regularised problem.
For later use, define
\begin{align*}
    \mathcal F_R[\rho]
      &\coloneqq\int_{B_R}\bigl(U(\rho)+V\rho\bigr)\,\dx, \\
      \mathcal F_{R,\ee}[\rho]
      &\coloneqq\int_{B_R}\bigl(U_\ee(\rho)+V\rho\bigr)\,\dx,
\end{align*}
the free-energy functional for the problem \eqref{eq:saturation} posed in the ball $B_R$ with no-flux and of its regularisation \eqref{eq:saturation regularised}. 

For fixed $R, \ee>0$, condition
\eqref{eq:regularised diffusion construction} makes
\eqref{eq:saturation regularised} a uniformly parabolic quasilinear
equation with the homogeneous no-flux boundary condition
\eqref{eq:regularised no flux}. This bounded-domain programme was carried
out for the present saturation structure in
\cite[Theorem 2.4]{Carrillo_FJ_Gomez-Castro25}. In particular it produces a \goodSemigroup{} $S^{R, \ee}$ for the regularised problem \eqref{eq:saturation regularised} over the class of admissible initial data $\mathcal{A}_R$ introduced in \eqref{def:AR}.


In particular, the \goodSemigroup{} $S^{R, \ee}$ is such that
\begin{align}
  & 0\leq S_t^{R,\ee}\rho_0^R\leq\alpha,
  \label{eq:regularised invariant region}\\
  & \|S_t^{R,\ee}\rho_0^R-S_t^{R,\ee}\eta_0^R\|_{L^1(B_R)}
  \leq\|\rho_0^R-\eta_0^R\|_{L^1(B_R)}.
  \label{eq:regularised bounded contraction}
\end{align}
For notational convenience, set
\[
 \rho_t^{R,\ee}=S_t^{R,\ee}\rho_0^R,
 \qquad
 \rho_t^{R,\ee,*}=\rho_t^{R,\ee}\chi_{B_R}.
\]
Mass conservation and \eqref{eq:regularised invariant region} imply that,
for every $1\leq p<+\infty$,
\begin{equation*}
 \|\rho_t^{R,\ee,*}\|_{L^p(\Rd)}
 \leq \alpha^{(p-1)/p}
       \|\rho_0^R\|_{L^1(B_R)}^{1/p}
 \leq \alpha^{(p-1)/p}
       \|\rho_0\|_{L^1(\Rd)}^{1/p},
\end{equation*}
while
$\|\rho_t^{R,\ee,*}\|_{L^\infty(\Rd)}\leq\alpha$.
The regularised semigroup also satisfies the free-energy dissipation, i.e. for any $0 \leq t_1 < t_2$ it follows that
\begin{equation*}
 0 \leq \int_{t_1}^{t_2}\int_{B_R}
 \mob_\ee(\rho_t^{R,\ee})
 \left|\nabla\bigl(U_\ee'(\rho_t^{R,\ee})+V\bigr)\right|^2
 \,\dx\,\dt \leq \mathcal F_{R,\ee}[\rho_{t_1}^{R,\ee}] - \mathcal F_{R,\ee}[\rho_{t_2}^{R,\ee}].
\end{equation*}

Now that the regularised problem and the semigroup $S^{R,\ee}$ have been
introduced, we specify the order of the two limiting procedures used in the
construction: We take the sequences $R_j$, $\ee_k$. First we pass to the limit $\ee_k \to 0$ for a fixed $R_j$
\[
    S^{R_j,\ee_k} \rho_0^{R_j} \xrightarrow[\ee_k\to0]{}S^{R_j} \rho_0^{R_j}
    \quad\text{on }B_{R_j},
\]
and afterwards we consider the second limit $R_j \to \infty$ 
\[
    \bigl(S_t^{R_j}\rho_0^{R_j} \bigr)^*\xrightarrow[R\to+\infty]{}S_t\rho_0
    \quad\text{on }\Rd.
\]
%
The estimates needed for the limit $R_j \to+\infty$ are local. Fix
$L>0$ and $R>2L$ and choose the local test function 
\begin{equation*}
    \chi_L\in C_c^\infty(B_{2L}) \text{ with }
0\leq\chi_L\leq1 \text{ and } \chi_L=1 \text{ on } B_L.     
\end{equation*}
In the following, we introduce some space and time local estimates. In order to do that first we introduce the auxiliary flow
\begin{equation}\label{eq:theta}
    \vartheta_\ee(s)=\int_0^s\Phi_\ee(\sigma)\,\diff\sigma.
\end{equation}

\begin{lemma}[Uniform local spatial Caccioppoli estimate]
    \label{lem:local spatial Caccioppoli estimate}
    Assume \ref{hyp:mobility}--\ref{hyp:Phi}. Let $\rho^{R,\ee}=S^{R,\ee}_\bullet\rho_0^R$ be the unique classical solution
    of \eqref{eq:saturation regularised} in the bounded domain $(t_1, t_2) \times B_R$. We have that
    \begin{equation}
     \int_{t_1}^{t_2}\int_{B_R}
       |\nabla\Phi_\ee(\rho^{R,\ee})|^2\chi_L^2\,\dx\,\dt
       \leq C_{L}.
       \label{eq:Caccioppoli Phi local}
    \end{equation}
    More precisely,
    \begin{equation*}
       C_{L}=C\left(
       \alpha, L, (t_2 - t_1),
       \|\Phi_\ee(\rho^{R,\ee})\|_
          {L^\infty((t_1,t_2)\times B_R)},
       \|\mob_\ee\|_{L^\infty(0,\alpha)},\|\nabla V\|_{L^2(B_{2L})},
       \|\nabla\chi_L\|_{L^2(\Rd)}
       \right).
    \end{equation*}
    This constant is independent of $R>2L$ and $\ee\in(0,1]$.
\end{lemma}

\begin{proof}
We take into account the auxiliary flow \eqref{eq:theta}. We compute and we get 
\begin{align*}
 \frac{\diff}{\dt}\int_{B_R}
      \vartheta_\ee(\rho^{R,\ee})\chi_L^2\,\dx
 &=-\int_{B_R}|\nabla\Phi_\ee(\rho^{R,\ee})|^2\chi_L^2\,\dx\\
 &\quad-\int_{B_R}\mob_\ee(\rho^{R,\ee})
       \nabla\Phi_\ee(\rho^{R,\ee})\cdot\nabla V\,\chi_L^2\,\dx\\
 &\quad-2\int_{B_R}\Phi_\ee(\rho^{R,\ee})
       \nabla\Phi_\ee(\rho^{R,\ee})\cdot\chi_L\nabla\chi_L\,\dx\\
 &\quad-2\int_{B_R}\Phi_\ee(\rho^{R,\ee})
       \mob_\ee(\rho^{R,\ee})\nabla V\cdot
       \chi_L\nabla\chi_L\,\dx.
\end{align*}
Young's inequality applied to the second and third terms yields
\begin{align*}
 \frac12\int_{B_R}|\nabla\Phi_\ee(\rho^{R,\ee})|^2
       \chi_L^2\,\dx
 &\leq-\frac{\diff}{\dt}\int_{B_R}
       \vartheta_\ee(\rho^{R,\ee})\chi_L^2\,\dx\\
 &\quad+C\|\mob_\ee\|_{L^\infty(0,\alpha)}^2
       \|\nabla V\|_{L^2(B_{2L})}^2\\
 &\quad+C\|\Phi_\ee\|_{L^\infty(0,\alpha)}^2
       \|\nabla\chi_L\|_{L^2(\Rd)}^2\\
 &\quad+2\|\Phi_\ee\|_{L^\infty(0,\alpha)}
       \|\mob_\ee\|_{L^\infty(0,\alpha)}
       \|\nabla V\|_{L^2(B_{2L})}
       \|\nabla\chi_L\|_{L^2(\Rd)}.
\end{align*}
Since $0\leq\rho^{R,\ee}\leq\alpha$,
\[
 \left|\int_{B_R}\vartheta_\ee(\rho_t^{R,\ee})
       \chi_L^2\,\dx\right|
 \leq\alpha|B_{2L}|\|\Phi_\ee\|_{L^\infty(0,\alpha)}.
\]
Integration over $(t_1,t_2)$ proves
\eqref{eq:Caccioppoli Phi local}. The uniformity in $R$ and $\ee$ follows
from \eqref{eq:regularisation convergence}.
\end{proof}

We can do the same for time estimates. We provide uniform Caccioppoli estimates for the time derivative.

\begin{lemma}[Uniform local time Caccioppoli estimate]
    \label{lem:local time Caccioppoli estimate}
    Assume \ref{hyp:mobility}--\ref{hyp:Phi}. Let $\rho^{R,\ee}=S^{R,\ee}_\bullet\rho_0^R$ be the unique classical solution
    of \eqref{eq:saturation regularised} in the bounded domain $(t_1, t_2) \times B_R$. We have that 
    \begin{align}
     \|\partial_t(\rho^{R,\ee} \chi_L^2)\|_
     {L^2(t_1,t_2;H^{-1}(B_{2L}))}^2
     &\leq C_{L},
     \label{eq:Caccioppoli time local}
    \end{align}
    The constant $C_{L}$ is independent
    of $R>2L$ and $\ee\in(0,1]$.
\end{lemma}

\begin{proof}
Set
\[
 J^{R,\ee}=\nabla\Phi_\ee(\rho^{R,\ee})
 +\mob_\ee(\rho^{R,\ee})\nabla V.
\]
For every $\zeta\in H_0^1(B_{2L})$, the weak formulation and
$\nabla(\chi_L^2\zeta)
=\chi_L^2\nabla\zeta+2\chi_L\zeta\nabla\chi_L$ give
\begin{align*}
 \left|\left\langle
  \partial_t(\rho^{R,\ee} \chi_L^2),\zeta
 \right\rangle_{H^{-1},H_0^1}\right|
 &\leq C_L\|J^{R,\ee} \chi_L\|_{L^2(B_{2L})}
          \|\zeta\|_{H^1(B_{2L})}.
\end{align*}
Moreover,
\begin{align*}
 \|J^{R,\ee} \chi_L\|_{L^2(B_{2L})}^2
 &\leq 2\int_{B_{2L}}
        |\nabla\Phi_\ee(\rho^{R,\ee})|^2\chi_L^2\,\dx
      +2\|\mob_\ee\|_{L^\infty(0,\alpha)}^2
        \|\nabla V\|_{L^2(B_{2L})}^2.
\end{align*}
Taking the dual norm, integrating in time, and applying
\eqref{eq:Caccioppoli Phi local} proves
\eqref{eq:Caccioppoli time local}. The constant is uniform in $R$ and
$\ee$ by \eqref{eq:regularisation convergence}.
\end{proof}

\subsection{The problem on balls $B_R$. The limit $\ee \rightarrow 0$}\label{sec:Limit ee to 0}

We will recall useful results from \cite{Carrillo_FJ_Gomez-Castro25} concerning the limit $\ee \to 0$ before taking the limit $R \to \infty$, i.e. on a fixed ball. The main novelty is that we will show stability of the Caccioppoli \textit{a priori} estimates. We start by stating \cite[Theorem 2.5]{Carrillo_FJ_Gomez-Castro25}.

\begin{lemma}[Convergence of the regularised problem]
    \label{lem:epsilon limit bounded ball}
    Assume \ref{hyp:mobility}--\ref{hyp:Phi} and fix $R>0$. Let us consider $\rho^{R, \ee}$ a solution to the regularised problem \eqref{eq:saturation regularised} with boundary condition \eqref{eq:regularised no flux} and $S^{R, \ee}$  its corresponding semigroup. Then, there exists a
    sequence $\ee_k\to0$ and a family
    $S^R=(S_t^R)_{t\geq0}$ on $\mathcal A_R$ such that, for every
    $\rho_0^R\in\mathcal A_R$ and $T>0$,
    \begin{equation}\label{eq:epsilon convergence bounded ball}
       S^{R,\ee_k}_\bullet\rho_0^R
       \longrightarrow S^R_\bullet\rho_0^R
       \quad\text{in }C([0,T];L^1(B_R))
       \quad\text{as }k\to+\infty.
    \end{equation}
    The trajectory $S^R_\bullet\rho_0^R$ is a weak solution of
    \eqref{eq:saturation} on $B_R$ with the no-flux boundary condition.
    Moreover, $S^R$ is a $C_0$-semigroup of $L^1(B_R)$-contractions and,
    for $0\leq t_1<t_2$,
    \begin{equation}\label{eq:bounded degenerate dissipation}
    \begin{split}
     \mathcal F_R[S_{t_2}^R\rho_0^R]
     &+\int_{t_1}^{t_2}\int_{B_R}
       \mob(\rho_t^R)
       |\nabla(U'(\rho_t^R)+V)|^2\,\dx\,\dt \leq\mathcal F_R[S_{t_1}^R\rho_0^R],
     \qquad \rho_t^R=S_t^R\rho_0^R.
    \end{split}
    \end{equation}
    It also satisfies the bounded-domain analogue of
    \eqref{eq:H-1 continuity}; hence $S^R$ is a \goodSemigroup{} on
    $B_R$.
\end{lemma}


We acknowledge that the regularity assumption on $\Phi$ presented in \eqref{eq:hypothesis Phi/Phi' ee = 0} is weaker than the corresponding assumption in \cite{Carrillo_FJ_Gomez-Castro25}. The proof presented here will only address this difficulty in order to explain its applicability to
the present assumptions and regularisation. 

\begin{proof}
We apply the proof of
\cite[Theorem 2.5]{Carrillo_FJ_Gomez-Castro25}
with $\Omega=B_R$. 

Although assumption \ref{hyp:Phi} imposes the quotient
bounds only on the corresponding half-intervals, these
are sufficient for the compactness argument in
\cite[Lemma 3.14]{Carrillo_FJ_Gomez-Castro25}.
Indeed, the two transformations used there involve
the factors
\[
 a_{1,\ee}(s)
 =[\Phi_\ee(\alpha)-\Phi_\ee(s)][\Phi_\ee(s)]_+,
 \qquad
 a_{2,\ee}(s)
 =[\Phi_\ee(s)-\Phi_\ee(0)][-\Phi_\ee(s)]_+.
\]
Since $\Phi_\ee(s_0)=0$ and $\Phi_\ee$ is increasing,
$a_{1,\ee}$ vanishes on $[0,s_0]$ and $a_{2,\ee}$
vanishes on $[s_0,\alpha]$. Consequently,
\eqref{eq:regularised ratio bound} gives precisely
\[
 |a_{i,\ee}(s)|\leq C\Phi_\ee'(s),
 \qquad i=1,2,
\]
uniformly in $\ee$ and $s$, as required there.

The remaining approximation properties used in the
cited proof are provided by
Lemma~\ref{lem:compatible regularisation} and
\eqref{eq:convergence regularised internal energies}.
Thus its compactness, contraction, and energy arguments
apply on the fixed ball $B_R$, yielding all the stated
conclusions.
\end{proof}

Once we have stated the convergence result of the regularised problem over balls from \cite{Carrillo_FJ_Gomez-Castro25} we proceed to discuss the stability of the local Caccioppoli estimates discussed in the previous subsection.

\begin{lemma}[Stability of the Caccioppoli estimates]
    \label{lem:degenerate local Caccioppoli estimates}
    Let 
    $\rho^R=S^R_\bullet\rho_0^R$ be the solution from Lemma \ref{lem:epsilon limit bounded ball}. For every
    $L>0$, $R>2L$, and $0\leq t_1<t_2 < +\infty$, it follows that
    \begin{align}
     \int_{t_1}^{t_2}\int_{B_{2L}}
     |\nabla\Phi(\rho^R)|^2\chi_L^2\,\dx\,\dt
     &\leq C_{L},
     \label{eq:Caccioppoli Phi degenerate local}\\
     \|\partial_t(\rho^R \chi_L^2)\|_
     {L^2(t_1,t_2;H^{-1}(B_{2L}))}
     &\leq C_{L},
     \label{eq:Caccioppoli time degenerate local}
    \end{align}
    where $C_{L}$ denotes the constants from Lemmas
    \ref{lem:local spatial Caccioppoli estimate} and
    \ref{lem:local time Caccioppoli estimate}.
\end{lemma}

\begin{proof}
We consider the sequence $\rho^{R, \ee_k}$ from Lemma \ref{lem:epsilon limit bounded ball}. 
We divide the proof into two steps.

\begin{enumeratesteps}
\step[Space estimates]
By Lemma \ref{lem:epsilon limit bounded ball}, in particular
\eqref{eq:epsilon convergence bounded ball},
$\rho^{R,\ee_k}\to\rho^R$ in $C([t_1,t_2];L^1(B_R))$. The uniform convergence
in \eqref{eq:regularisation convergence}, the uniform boundedness of
$\Phi_{\ee_k}$ on $[0,\alpha]$, and the uniform continuity of $\Phi$ give
\begin{equation}\label{eq:Phi epsilon strong local}
 \Phi_{\ee_k}(\rho^{R,\ee_k})
 \longrightarrow\Phi(\rho^R)
 \quad\text{strongly in }L^2( (t_1, t_2) \times B_{R}).
\end{equation}
Indeed, the convergence first holds in measure, while the invariant-region
estimate \eqref{eq:regularised invariant region} and
\eqref{eq:regularisation convergence} provide a uniform
$L^\infty((t_1, t_2) \times B_{R})$ bound.

Estimate \eqref{eq:Caccioppoli Phi local} makes
$\nabla\Phi_{\ee_k}(\rho^{R,\ee_k})\chi_L$ weakly relatively compact in
$L^2((t_1, t_2) \times B_{R})$. Therefore, from the Banach-Alaoglu Theorem there exists $G \in L^2( (t_1, t_2) \times B_R)$ such that (up to a subsequence)
\begin{equation*}
    \nabla\Phi_{\ee_k}(\rho^{R,\ee_k})\chi_L \rightharpoonup G \quad \text{weakly in } L^2( (t_1, t_2) \times B_R). 
\end{equation*}
Thanks to 
\eqref{eq:Phi epsilon strong local}, in the sense of
distributions on $(t_1, t_2) \times B_{R}$,
\[
 G=\chi_L\nabla\Phi(\rho^R).
\]
Consequently $\Phi(\rho^R)\in L^2(t_1,t_2;H^1(B_L))$, and weak lower
semicontinuity in \eqref{eq:Caccioppoli Phi local} yields
\eqref{eq:Caccioppoli Phi degenerate local}. 

\step[Time estimates]
By \eqref{eq:Caccioppoli time local} and the Banach--Alaoglu theorem,
there exists $z\in L^2(t_1,t_2;H^{-1}(B_{2L}))$ such that, up to a
subsequence,
\[
 \partial_t(\rho^{R,\ee_k} \chi_L^2)
 \rightharpoonup z
 \quad\text{weakly in }
 L^2(t_1,t_2;H^{-1}(B_{2L})).
\]
On the other hand, the convergence \eqref{eq:epsilon convergence bounded ball} implies
$\rho^{R,\ee_k} \chi_L^2 \to \rho^R \chi_L^2$ in
$C([0,t_2];L^1(B_{2L}))$. Testing the preceding weak convergence against
$C_c^\infty((t_1,t_2)\times B_{2L})$ therefore identifies
$z=\partial_t(\rho^R \chi_L^2)$. Weak lower semicontinuity in
\eqref{eq:Caccioppoli time local} proves
\eqref{eq:Caccioppoli time degenerate local} with the same constant.
\end{enumeratesteps}
\end{proof}

\subsection{The limit $R \to + \infty$. Local convergence}\label{sec:Existence local}

We devote this subsection to study the local convergence as we take expanding balls $B_R$ i.e. only with the set of assumptions \ref{hyp:mobility}--\ref{hyp:Phi} and before taking into account \ref{hyp:free energy bdd below}--\ref{hyp:growth of V}. In order to do that we start by 
defining the extended bounded-ball evolution
\begin{equation*}
    \mathcal S_t^{R,\ast} \rho_0
    \coloneqq\left(S_t^R\rho_0^R\right)\chi_{B_R}.
\end{equation*}
Since $S^R$ is $L^1$-contractive, due to Lemma \ref{lem:epsilon limit bounded ball} it follows that
\begin{equation}\label{eq:contraction R degenerate}
    \|\mathcal S_t^{R,\ast}\rho_0-\mathcal S_t^{R,\ast} \eta_0\|_{L^1(\Rd)}
    \leq\|\rho_0-\eta_0\|_{L^1(\Rd)}.
\end{equation}

Let us also introduce some nonlinear transformations already presented in
\cite[Lemma 3.14]{Carrillo_FJ_Gomez-Castro25}. 
With $s_0$ as in \eqref{eq:s0}, we define
\begin{align}
 \Psi_1(s)
 &=\int_{s_0}^{\max\{s,s_0\}}
   [\Phi(\alpha)-\Phi(\sigma)]\,[\Phi(\sigma)]_+
   \,\diff\sigma,
 \label{eq:Psi 1 degenerate definition}\\
 \Psi_2(s)
 &=\int_0^{\min\{s,s_0\}}
   [\Phi(\sigma)-\Phi(0)]\,[-\Phi(\sigma)]_+
   \,\diff\sigma.
 \label{eq:Psi 2 degenerate definition}
\end{align}
We take advantage of these two nonlinear transformations in order to formulate the resulting
compactness criterion locally and independently of the approximation
procedure. 

\begin{proposition}
    \label{lem:local nonlinear compactness}
    Assume \ref{hyp:mobility}--\ref{hyp:Phi}. Let
    $\Omega_k\subset\Rd$ be open sets such that, for every $L>0$,
    $B_{2L}\subset\Omega_k$ for all sufficiently large $k$, and let
    \[
        \rho_k\colon[0,+\infty)\times\Omega_k\longrightarrow[0,\alpha].
    \]
    Suppose that, for every $L>0$ and a cut-off
    $\chi_L\in C_c^\infty(B_{2L})$ satisfying $\chi_L=1$ on $B_L$,
    \begin{align}
        \sup_k\sum_{i=1}^2
        \|\Psi_i(\rho_k)\|_{L^2(0,T;H^1(B_L))}
        &<+\infty,
        \label{eq:abstract nonlinear spatial bound}\\
        \sup_k
        \|\partial_t(\rho_k \chi_L^2)\|_
        {L^2(0,T;H^{-1}(B_{2L}))}
        &<+\infty.
        \label{eq:abstract nonlinear time bound}
    \end{align}
    Then, there exist a subsequence, not relabelled, and a function
    $\rho\colon[0,+\infty)\times\Rd\to[0,\alpha]$ such that, for every
    $T, L>0$ and $1\leq p<+\infty$,
    \begin{align}
        \rho_k&\longrightarrow\rho
        &&\text{in }C([0,T];H^{-1}(B_L)),
        \\
        \rho_k&\longrightarrow\rho
        &&\text{strongly in }L^p((0,T)\times B_L).
    \end{align}
\end{proposition}

\begin{proof} 

Let us fix $T,L>0$. We divide the proof into several steps.

\begin{enumeratesteps}

\step[$H^{-1}$ precompactness]
Since $0\leq\rho_k\leq\alpha$,
\[
    \sup_k
    \|\rho_k\|_{L^\infty(0,T;L^2(B_{2L}))}
    \leq \alpha|B_{2L}|^{1/2}.
\]
Applying \eqref{eq:abstract nonlinear time bound} with
the cutoff $\chi_{2L}$, and testing with zero extensions
of functions in $H_0^1(B_{2L})$, gives
\[
    \sup_k
    \|\partial_t\rho_k\|_
    {L^2(0,T;H^{-1}(B_{2L}))}
    \leq C_{L}.
\]
Consequently,
\[
    \|\rho_k(t)-\rho_k(s)\|_{H^{-1}(B_{2L})}
    \leq C_{L}|t-s|^{1/2},
    \qquad s,t\in[0,T].
\]
Since $L^2(B_{2L})\Subset H^{-1}(B_{2L})$, the
Arzel\`a--Ascoli theorem provides a subsequence and a limit
$\rho$ such that
\[
    \rho_k\longrightarrow\rho
    \quad\text{in }C([0,T];H^{-1}(B_{2L})).
\]

\step[Application of the Aubin--Lions lemma]
We first establish strong compactness of the transformed
densities $\Psi_i(\rho_k)$.
Each $\Psi_i$ is nondecreasing and Lipschitz on
$[0,\alpha]$.
Indeed
\[
    0\leq \Psi_1'(s) = [\Phi (\alpha) - \Phi(s)] [\Phi (s)]_+ \leq 2 \| \Phi \|_{L^\infty (0,\alpha)}^2
\]
and $\Phi$ is bounded due to \ref{hyp:Phi} since $\mob U'' \in L^1(0,\alpha)$. 
Writing $K_i=\operatorname{Lip}(\Psi_i)$,
we therefore have
\[
    |\Psi_i(a)-\Psi_i(b)|^2
    \leq K_i(a-b)\bigl(\Psi_i(a)-\Psi_i(b)\bigr),
    \qquad a,b\in[0,\alpha].
\]
Moreover, 
for $0<h<T$, write $\Delta_h f(t)=f(t+h)-f(t)$.
The time-derivative bound established in Step~1 gives
\[
    \|\Delta_h\rho_k\|_
    {L^2(0,T-h;H^{-1}(B_{2L}))}
    \leq C_{L}h.
\]
Furthermore, \eqref{eq:abstract nonlinear spatial bound},
applied on $B_{2L}$, implies that
\[
    \chi_L^2\Delta_h\Psi_i(\rho_k)
\]
is uniformly bounded in
$L^2(0,T-h;H_0^1(B_{2L}))$, independently of $k$ and $h$.
Using the preceding scalar inequality, $\chi_L=1$ on
$B_L$, and duality, we obtain
\begin{align*}
    &\|\Delta_h\Psi_i(\rho_k)\|_
    {L^2((0,T-h)\times B_L)}^2\\
    &\quad\leq
    K_i\int_0^{T-h}
    \left\langle
        \Delta_h\rho_k,\,
        \chi_L^2\Delta_h\Psi_i(\rho_k)
    \right\rangle_{H^{-1}(B_{2L}),H_0^1(B_{2L})}
    \,\dt\\
    &\quad\leq
    C_{L}
    \|\Delta_h\rho_k\|_
    {L^2(0,T-h;H^{-1}(B_{2L}))}
    \leq C_{L}h.
\end{align*}
Thus the time translations of $\Psi_i(\rho_k)$ tend to
zero uniformly in $k$ in the local space--time $L^2$ norm.

Together with \eqref{eq:abstract nonlinear spatial bound}
and the compact embedding $H^1(B_L)\Subset L^2(B_L)$,
the time-translation form of the Aubin--Lions lemma gives
relative compactness of $\Psi_i(\rho_k)$ in
$L^2((0,T)\times B_L)$.
Up to a further subsequence, there exist $w_1,w_2$ such that
\[
    \Psi_i(\rho_k)\longrightarrow w_i
    \quad\text{strongly in }L^2((0,T)\times B_L)
    \quad\text{and almost everywhere},
    \qquad i=1,2.
\]

\step[Characterisation of the limit]
The strict monotonicity of $\Phi$ and definitions
\eqref{eq:Psi 1 degenerate definition}--\eqref{eq:Psi 2 degenerate definition} show that $\Psi_1$
is strictly increasing on $[s_0,\alpha]$ and constant on
$[0,s_0]$, whereas $\Psi_2$ is strictly increasing on
$[0,s_0]$ and constant on $[s_0,\alpha]$. 
Set
\[
    A_1=\Psi_1(\alpha)>0,
    \qquad
    A_2=\Psi_2(s_0)=\Psi_2(\alpha)>0,
\]
and define the continuous inverse branches
\[
    \varpi_1(r)
    =\bigl(\Psi_1|_{[s_0,\alpha]}\bigr)^{-1}(A_1r),
    \qquad
    \varpi_2(r)
    =\bigl(\Psi_2|_{[0,s_0]}\bigr)^{-1}(A_2r),
    \qquad r\in[0,1].
\]
By construction,
\[
    \varpi_1\!\left(\frac{\Psi_1(s)}{A_1}\right)
    =\max\{s,s_0\},
    \qquad
    \varpi_2\!\left(\frac{\Psi_2(s)}{A_2}\right)
    =\min\{s,s_0\}.
\]
Hence, for every $s\in[0,\alpha]$,
\[
    s=
    \varpi_1\!\left(\frac{\Psi_1(s)}{A_1}\right)
    +\varpi_2\!\left(\frac{\Psi_2(s)}{A_2}\right)
    -s_0.
\]
Applying this identity to $\rho_k$ and using Step~2
and continuity of $\varpi_i$ gives
\[
    \rho_k\longrightarrow
    \widetilde\rho
    \coloneqq
    \varpi_1\!\left(\frac{w_1}{A_1}\right)
    +\varpi_2\!\left(\frac{w_2}{A_2}\right)
    -s_0
    \quad\text{almost everywhere on }(0,T)\times B_L.
\]
Since $0\leq\rho_k\leq\alpha$, the same bound holds for
$\widetilde\rho$, and the Dominated Convergence Theorem yields
\[
    \rho_k\longrightarrow\widetilde\rho
    \quad\text{strongly in }L^p((0,T)\times B_L),
    \qquad 1\leq p<+\infty.
\]
This strong limit agrees distributionally with the limit
$\rho$ obtained in Step~1. Therefore
$\widetilde\rho=\rho$, proving both stated convergences.
Continuity of $\Psi_i$ also identifies $w_i=\Psi_i(\rho)$.

\end{enumeratesteps}

\end{proof}

This result can be applied to our particular case and leads to the following corollary. 

\begin{corollary}[Local convergence]
    \label{cor:R limit local solutions}
    Assume \ref{hyp:mobility}--\ref{hyp:Phi}. There exist a sequence
    $R_k\to+\infty$, independent of the initial datum, and maps
    $S_t\colon\mathcal A\to\mathcal A$ such that, for every
    $\rho_0\in\mathcal A$, $T,L>0$, and $1\leq p<+\infty$,
    \begin{align}
       \mathcal S^{R_k,\ast}_\bullet\rho_0
       &\longrightarrow S_\bullet\rho_0
       &&\text{in }C([0,T];H^{-1}(B_L)),
       \label{eq:R limit local H-1}\\
       \mathcal S^{R_k,\ast}_\bullet\rho_0
       &\longrightarrow S_\bullet\rho_0
       &&\text{strongly in }L^p((0,T)\times B_L).
       \label{eq:R limit local Lp}
    \end{align}
    Moreover, $S_0\rho_0=\rho_0$, the trajectory $S_\bullet\rho_0$ is a
    weak solution of \eqref{eq:saturation}, and
    $0\leq S_t\rho_0\leq\alpha$. The maps $S_t$ obey
    \eqref{eq:main local mass inequality} and
    \eqref{eq:main local contraction} for every $t\geq0$. Finally,
    for every $0\leq t_1<t_2\leq T$,
    \begin{align}
      \int_{t_1}^{t_2}\int_{B_{2L}}
      |\nabla\Phi(S_t\rho_0)|^2\chi_L^2\,\dx\,\dt
      &\leq C_{L},
      \label{eq:Caccioppoli Phi whole space}\\
      \|\partial_t(\chi_L^2S_\bullet\rho_0)\|_
      {L^2(t_1,t_2;H^{-1}(B_{2L}))}
      &\leq C_{L}.
      \label{eq:Caccioppoli time whole space}
    \end{align}
\end{corollary}

\begin{proof}
Let $R_k\to+\infty$, fix $\rho_0\in\mathcal A$, and set
\[
    \Omega_k\coloneqq B_{R_k},
    \qquad
    u_k\coloneqq\mathcal S^{R_k,\ast}_\bullet\rho_0.
\]
We apply Proposition \ref{lem:local nonlinear compactness}
to the restrictions of $u_k$ to
$[0,+\infty)\times\Omega_k$.
Since $R_k\to+\infty$, every fixed ball $B_{2L}$ is
contained in $\Omega_k$ for all sufficiently large $k$. Let us fix $T,L>0$. In the following we check that the assumptions of Proposition \ref{lem:local nonlinear compactness} are fulfilled and we apply it.

\begin{enumeratesteps}

\step[Spatial compactness
\eqref{eq:abstract nonlinear spatial bound}]
The definitions of $\Psi_1$ and $\Psi_2$,
together with assumption \ref{hyp:Phi}, give
\[
    0\leq\Psi_i'(s)\leq C\Phi'(s),
    \qquad s\in(0,\alpha),\quad i=1,2.
\]
Integrating between two points and using strict monotonicity
of $\Phi$ shows that $\Psi_i\circ\Phi^{-1}$ is Lipschitz
on $\Phi([0,\alpha])$.
Since $\Phi(u_k)\in L^2(0,T;H^1(B_L))$, the Sobolev chain
rule yields
\[
    |\nabla\Psi_i(u_k)|
    \leq C|\nabla\Phi(u_k)|
    \quad\text{a.e. on }(0,T)\times B_L.
\]
Using boundedness of $\Psi_i$ on $[0,\alpha]$,
$\chi_L=1$ on $B_L$, and
\eqref{eq:Caccioppoli Phi degenerate local}, we obtain
\begin{align*}
    \sum_{i=1}^2
    \|\Psi_i(u_k)\|_{L^2(0,T;H^1(B_L))}^2
    &\leq
    T|B_L|\sum_{i=1}^2
    \|\Psi_i\|_{L^\infty(0,\alpha)}^2\\
    &\quad+
    C\int_0^T\int_{B_{2L}}
    \chi_L^2|\nabla\Phi(u_k)|^2\,\dx\,\dt\\
    &\leq C_{T,L}.
\end{align*}
The constant is independent of $k$, proving
\eqref{eq:abstract nonlinear spatial bound} from Proposition \ref{lem:local nonlinear compactness} for all
sufficiently large $k$.

\step[Time compactness
\eqref{eq:abstract nonlinear time bound}]
For $R_k>2L$, estimate
\eqref{eq:Caccioppoli time degenerate local} gives directly
\[
    \sup_{k:R_k>2L}
    \|\partial_t(\chi_L^2u_k)\|_
    {L^2(0,T;H^{-1}(B_{2L}))}
    \leq C_{T,L}.
\]
The zero extension does not affect this estimate, since
$\chi_L$ is compactly supported in $B_{2L}\subset\Omega_k$.
This proves \eqref{eq:abstract nonlinear time bound} from Proposition \ref{lem:local nonlinear compactness}.

\step[Application of Proposition
\ref{lem:local nonlinear compactness}]
Proposition \ref{lem:local nonlinear compactness} gives
both asserted local convergences along a subsequence for
each fixed initial datum.

A diagonal argument on a countable dense subset of
$\mathcal A$, followed by extension through the bounded-ball
contraction estimate and the saturation bound, makes the
sequence independent of the initial datum.
Up to a subsequence and relabelling, we can write
\[
    u_k^\eta=\mathcal S^{R_k,\ast}_\bullet\eta,
    \qquad
    u^\eta=S_\bullet\eta,
    \qquad \eta\in\mathcal A.
\]
The limits agree on overlapping balls and time intervals and hence they satisfy 
\begin{align*}
       u_k^\eta
       &\longrightarrow u^\eta
       &&\text{in }C([0,T];H^{-1}(B_L)),
       \\
       u_k^\eta
       &\longrightarrow u^\eta
       &&\text{strongly in }L^p((0,T)\times B_L).
    \end{align*}
In the following we prove the remaining properties of the constructed limit. 

\begin{enumeratesteps}

\step[$L^1$ bounds, $L^1$ contraction and weak solution]
Convergence in $H^{-1}(B_L)$
and uniform boundedness of the densities imply
\[
    u_k^\eta(t)\rightharpoonup u^\eta(t)
    \quad\text{weakly in }L^2(B_L)
\]
for every $\eta\in\mathcal A$. In particular,
$0\leq u^\eta(t)\leq\alpha$.
Weak lower semicontinuity of the $L^1(B_L)$ norm,
bounded-ball mass conservation, and
\eqref{eq:contraction R degenerate} give
\begin{align*}
    \|S_t\rho_0\|_{L^1(B_L)}
    &\leq
    \liminf_{k\to+\infty}
    \|u_k^{\rho_0}(t)\|_{L^1(B_L)}
    \leq \|\rho_0\|_{L^1(\Rd)},\\
    \|S_t\rho_0-S_t\eta_0\|_{L^1(B_L)}
    &\leq
    \liminf_{k\to+\infty}
    \|u_k^{\rho_0}(t)-u_k^{\eta_0}(t)\|_{L^1(B_L)} \leq \|\rho_0-\eta_0\|_{L^1(\Rd)}.
\end{align*}
Letting $L\to+\infty$ proves
\eqref{eq:main local mass inequality} and
\eqref{eq:main local contraction}.
Consequently, $S_t$ maps $\mathcal A$ into itself
for every $t\geq0$.

\step[Spatial Caccioppoli estimate \eqref{eq:Caccioppoli Phi whole space}]
Fix $\rho_0\in\mathcal A$ and write
$u_k=u_k^{\rho_0}$ and $u=u^{\rho_0}$.
For $0\leq t_1<t_2$ and $L>0$, continuity of $\Phi$,
strong local convergence, and the saturation bound give
\[
    \Phi(u_k)\longrightarrow\Phi(u)
    \quad\text{strongly in }
    L^2((t_1,t_2)\times B_{2L}).
\]
Applying the spatial estimate with a larger cutoff
if necessary bounds the gradients on $B_{2L}$.
The preceding strong convergence identifies their weak
limit, so
\[
    \chi_L\nabla\Phi(u_k)
    \rightharpoonup\chi_L\nabla\Phi(u)
    \quad\text{weakly in }
    L^2((t_1,t_2)\times B_{2L}).
\]
Weak lower semicontinuity and
\eqref{eq:Caccioppoli Phi degenerate local} therefore yield
\begin{align*}
    \int_{t_1}^{t_2}\int_{B_{2L}}
    \chi_L^2|\nabla\Phi(u)|^2\,\dx\,\dt
    &\leq
    \liminf_{k\to+\infty}
    \int_{t_1}^{t_2}\int_{B_{2L}}
    \chi_L^2|\nabla\Phi(u_k)|^2\,\dx\,\dt \leq C_L,
\end{align*}
which proves \eqref{eq:Caccioppoli Phi whole space}.

\step[Time Caccioppoli estimate \eqref{eq:Caccioppoli time whole space}]
Estimate \eqref{eq:Caccioppoli time degenerate local}
gives weak compactness of
$\partial_t(\chi_L^2u_k)$ in the space 
$L^2(t_1,t_2;H^{-1}(B_{2L}))$.
Strong local convergence of $u_k$ identifies every weak
limit distributionally as $\partial_t(\chi_L^2u)$.
Consequently, weak lower semicontinuity gives
\begin{align*}
    \|\partial_t(\chi_L^2u)\|_
    {L^2(t_1,t_2;H^{-1}(B_{2L}))}
    &\leq
    \liminf_{k\to+\infty}
    \|\partial_t(\chi_L^2u_k)\|_
    {L^2(t_1,t_2;H^{-1}(B_{2L}))} \leq C_L.
\end{align*}
This proves \eqref{eq:Caccioppoli time whole space}.

\step[Weak solution to \eqref{eq:saturation}]
Finally, continuity of the mobility $\mob$ and strong local convergence
give
\[
    \mob(u_k)\nabla V\longrightarrow\mob(u)\nabla V
    \quad\text{strongly in }L^2((0,T)\times B_L).
\]
Together with the weak gradient convergence established
above, this allows passage to the bounded-ball weak
formulation against compactly supported test functions.
Moreover, convergence in $C([0,T];H^{-1}(B_L))$ and
$u_k(0)=\rho_0\chi_{B_{R_k}}$ give
$S_0\rho_0=\rho_0$.
Thus $S_\bullet\rho_0$ is a weak solution of
\eqref{eq:saturation} with initial datum $\rho_0$.
\end{enumeratesteps}
\end{enumeratesteps}
\end{proof}

\subsection{The limit $R \to + \infty$. Global convergence}\label{sec:Existence global}

\label{sec:energy bdd below}

Finally, taking into account the extra compactness assumptions \ref{hyp:free energy bdd below}--\ref{hyp:growth of V} we can improve the convergence result to global. In order to do that we first need to explore the properties and convergence of the free-energy functional.  
We start by recovering a local formulation of
\cite[Lemma 3.9]{Carrillo_FJ_Gomez-Castro25} concerning the free-energy dissipation. For $0\leq s\leq\alpha$ and $z\in\mathbb R^d$, set
\begin{equation}\label{eq:dissipation integrand convention}
 \mathfrak D (s,z)
 \coloneqq
 \begin{cases}
  |z|^2/\mob(s),&\mob(s)>0,\\
  0,&\mob(s)=0\text{ and }z=0,\\
  +\infty,&\mob(s)=0\text{ and }z\neq 0
 \end{cases}
\end{equation}
Thus, for
$z=\nabla\Phi(\rho)+\mob(\rho)\nabla V$, the quantity
$\mathfrak D(\rho,z)$ is the rigorous interpretation of 
$\mob(\rho)|\nabla(U'(\rho)+V)|^2$.

\begin{lemma}[Lower semicontinuity of the dissipation]
 \label{lem:dissipation lower semicontinuity}
 Let $0\leq t_1<t_2<+\infty$ and let $Q\subset\Rd$ have finite
 measure. Suppose that $0\leq\rho_k,\rho\leq\alpha$ and
 \[
  \rho_k\longrightarrow\rho\quad\text{a.e. in }(t_1,t_2)\times Q,
  \qquad
  z_k\rightharpoonup z\quad\text{weakly in }
  L^2((t_1,t_2)\times Q;\mathbb R^d).
 \]
 Then
 \begin{equation}\label{eq:dissipation lower semicontinuity}
  \int_{t_1}^{t_2}\int_Q\mathfrak D(\rho,z)\,\dx\,\dt
  \leq\liminf_{k\to+\infty}
  \int_{t_1}^{t_2}\int_Q\mathfrak D(\rho_k,z_k)\,\dx\,\dt.
 \end{equation}
\end{lemma}

\begin{proof}
Since $Q$ has finite measure and $\rho_k$ is dominated by $\alpha$, the continuity of $\mob$ and the Dominated Convergence Theorem give
$\mob(\rho_k)\to\mob(\rho)$ strongly in every finite $L^p$ on $(t_1,t_2) \times Q$. From the definition of $\mathfrak D$ we observe that 
\[
 \mathfrak D(s,z)=\sup_{\xi\in\mathbb R^d}
 \bigl(2z\cdot\xi-\mob(s)|\xi|^2\bigr).
\]
Therefore, for every bounded measurable vector field $\xi$,
\begin{align*}
 \liminf_{k\to+\infty}
 \int_{t_1}^{t_2}\int_Q\mathfrak D(\rho_k,z_k)\,\dx\,\dt
 &\geq\int_{t_1}^{t_2}\int_Q
 \bigl(2z\cdot\xi-\mob(\rho)|\xi|^2\bigr)\,\dx\,\dt.
\end{align*}
Taking the supremum over such fields $\xi$ and using truncation to recover the
pointwise supremum proves \eqref{eq:dissipation lower semicontinuity}.
\end{proof}

Assumption \ref{hyp:free energy bdd below} prevents the negative internal
energy from cancelling the confinement term. Since $U$ is convex,
$U(0)=0$, and $U$ is continuous on $[0,\alpha]$, there is $C_U>0$ such
that
\begin{equation}\label{eq:positive energy linear control}
    [U(s)]_+\leq C_Us,\qquad 0\leq s\leq\alpha.
\end{equation}
Moreover, \eqref{eq:negative internal energy confinement} gives
\begin{equation}\label{eq:energy coercive lower bound}
    \mathcal F[\rho]
    \geq(1-\vartheta)\int_{\Rd}V\rho\,\dx-\|g\|_{L^1(\Rd)}.
\end{equation}
In particular, energy sublevels control the first confinement moment.
The proof below uses this control for mass tightness.

\begin{lemma}[Lower semicontinuity and recovery for expanding balls]
    \label{lem:energy expanding balls}
    Assume \ref{hyp:free energy bdd below}--\ref{hyp:growth of V}.
    Let $R_k\to+\infty$, let $\rho_k\in\mathcal A_{R_k}$, and consider
    $\rho_k^\ast$ its extension by zero. If $\rho\in\mathcal A$ and $\rho_k^*\to\rho$ in
    $L^1_{\mathrm{loc}}(\Rd)$, then
    \begin{equation}\label{eq:energy liminf R}
       \mathcal F[\rho]
       \leq\liminf_{k\to\infty}\mathcal F_{R_k}[\rho_k].
    \end{equation}
    Conversely, every $\rho\in\mathcal A_{\mathcal F}$ admits
    $\rho_k\in\mathcal A_{R_k}$ such that
    \begin{equation}\label{eq:energy recovery R}
       \rho_k^*\to\rho\quad\text{in }L^1(\Rd),\qquad
       \mathcal F_{R_k}[\rho_k]\to\mathcal F[\rho].
    \end{equation}
    The recovery sequence may be chosen to preserve the mass for all
    sufficiently large $k$.
\end{lemma}

\begin{proof}
Take a subsequence realizing the lower limit; if it is infinite, there is
nothing to prove. A further subsequence converges almost everywhere
locally. Since $U(0)=0$, the ball energy equals the energy of the zero
extension, and
\[
 U(\rho_k^*)+V\rho_k^*+g\geq0.
\]
Fatou's lemma on $\Rd$ proves \eqref{eq:energy liminf R}.

For recovery, truncation alone gives the convergence in
\eqref{eq:energy recovery R}. To preserve mass, choose a bounded set $K$
such that $c_K=\int_K(\alpha-\rho)\,\dx>0$. Put
$a_k=\int_{\Rd\setminus B_{R_k}}\rho\,\dx$ and, for large $k$ with
$K\subset B_{R_k}$ and $a_k\leq c_K$, set
\[
 \rho_k^*=\rho\chi_{B_{R_k}}
             +\frac{a_k}{c_K}(\alpha-\rho)\chi_K.
\]
This preserves the invariant region and the mass, and its $L^1$ error is
$2a_k\to0$. On $K$, the correction tends to zero uniformly, so uniform
continuity of $U$ and boundedness of $V$ give convergence of the corrected
energy. The omitted tails tend to zero by absolute integrability of
$U(\rho)$ and $V\rho$.
\end{proof}

We now use assumptions \ref{hyp:free energy bdd below}--\ref{hyp:growth of V}
to obtain uniform mass tightness for families with bounded mass and free energy.
This tail control will later allow us to upgrade local convergence to global
$L^1(\Rd)$ convergence.

\begin{lemma}[Mass tightness]
    \label{lem:uniform tightness omega}
    Assume \ref{hyp:free energy bdd below}--\ref{hyp:growth of V}. Let
    $\rho_j $ be a sequence in $\mathcal A$ with  uniformly bounded masses and
    free energies. Then
    \begin{equation}\label{eq:uniform confinement moment omega}
        \sup_j\int_{\Rd}\bigl(V\rho_j+|U(\rho_j)|\bigr)\,\dx<+\infty.
    \end{equation}
    If $\nu_L=\inf_{\Rd\setminus B_L}V$, then, for sufficiently large
    $L$,
    \begin{equation}\label{eq:tail from V moment}
       \sup_j\int_{\Rd\setminus B_L}\rho_j\,\dx
       \leq\frac C{\nu_L}\longrightarrow0.
    \end{equation}
    In particular,
    \begin{equation}\label{eq:uniform tightness omega}
       \lim_{L\to+\infty}\sup_j
       \int_{\Rd\setminus B_L}\rho_j\,\dx=0.
    \end{equation}
\end{lemma}

\begin{proof}
Estimate \eqref{eq:energy coercive lower bound} gives a uniform bound on
$\int V\rho_j$. Equations \eqref{eq:positive energy linear control} and
\eqref{eq:negative internal energy confinement} then give
\[
 \int_{\Rd}|U(\rho_j)|\,\dx
 \leq C_U\|\rho_j\|_1+\vartheta\int_{\Rd}V\rho_j\,\dx+\|g\|_1,
\]
which proves \eqref{eq:uniform confinement moment omega}. Finally,
\[
 \int_{\Rd\setminus B_L}\rho_j\,\dx
 \leq\frac1{\nu_L}\int_{\Rd}V\rho_j\,\dx,
\]
and coercivity proves the tail estimates.
\end{proof}

Let us now apply the uniform mass tightness to our \goodSemigroup{}.

\begin{lemma}
    \label{lem:tails degenerate ball semigroups}
    Assume \ref{hyp:free energy bdd below}--\ref{hyp:growth of V}. Then,
    for every $\rho_0\in\mathcal A$ and $T>0$,
    \begin{equation}\label{eq:uniform mass tails R}
       \lim_{L\to+\infty}\sup_{R>L}\sup_{t\in[0,T]}
       \int_{B_R\setminus B_L}S_t^R\rho_0^R\,\dx=0.
    \end{equation}
    If $\rho_0\in\mathcal A_{\mathcal F}$, then
    \begin{equation}\label{eq:uniform ball confinement moment}
       \sup_R\sup_{t\geq0}\int_{B_R}
       \bigl(V\rho_t^R+|U(\rho_t^R)|\bigr)\,\dx<+\infty,
       \qquad\rho_t^R=S_t^R\rho_0^R.
    \end{equation}
\end{lemma}

\begin{proof} We divide the proof in several steps. 
\begin{enumeratesteps}
\step[Compactly supported data]
Let $\eta\in\mathcal A_{\mathcal F}$ be compactly supported. For every
$R$ such that $B_R$ contains its support, mass conservation and
\eqref{eq:bounded degenerate dissipation} give
\begin{equation}\label{eq:ball energy sublevel estimate}
 \|S_t^R\eta\|_{L^1(B_R)}=\|\eta\|_{L^1(\Rd)},\qquad
 \mathcal F_R[S_t^R\eta]\leq\mathcal F[\eta].
\end{equation}
Because $U(0)=0$, the energy of the zero extension equals its
energy inside the ball $B_R$. Lemma \ref{lem:uniform tightness omega} therefore yields uniform
mass tails for this family, independently of $R$ and $t\geq0$.

\step[Arbitrary data]
Let $\rho_0\in\mathcal A$ and choose compactly supported
$\eta\in\mathcal A_{\mathcal F}$ with
$\|\rho_0-\eta\|_{L^1(\Rd)}<\delta$. For $L$ containing its support
and $R>L$, the contraction estimate
\eqref{eq:contraction R degenerate} gives
\begin{align}
 \int_{B_R\setminus B_L}S_t^R\rho_0^R\,\dx
 &\leq\|S_t^R\rho_0^R-S_t^R\eta\|_{L^1(B_R)}
       +\int_{B_R\setminus B_L}S_t^R\eta\,\dx \leq\delta+\int_{B_R\setminus B_L}S_t^R\eta\,\dx.
 \label{eq:mass tail contraction extension}
\end{align}
For fixed $\delta>0$, the approximation $\eta$ is fixed.
By \eqref{eq:ball energy sublevel estimate} and
\eqref{eq:tail from V moment}, the last integral in
\eqref{eq:mass tail contraction extension} tends to zero
as $L\to+\infty$, uniformly for $R>L$ and $t\in[0,T]$.
Taking the suprema in
\eqref{eq:mass tail contraction extension} therefore gives
a limit superior bounded by $\delta$.
Since $\delta>0$ is arbitrary, this proves
\eqref{eq:uniform mass tails R}.


For $\rho_0\in\mathcal A_{\mathcal F}$, we have that $\sup_R\mathcal F_R[\rho_0^R]<+\infty$ and hence
\begin{equation*}
 \mathcal F_R[S_t^R\rho_0^R]\leq\mathcal F_R[\rho_0^R] < + \infty .
\end{equation*}
Mass conservation and \eqref{eq:uniform confinement moment omega},
applied to the zero extensions, prove
\eqref{eq:uniform ball confinement moment}.
\end{enumeratesteps}
\end{proof}

The next lemma supplies uniform time modulus over a well-chosen dense subset of the admissible initial data $\mathcal{A}$. We need it in order to pass the
semigroup identity through the limit $R\to+\infty$. 

\begin{lemma}[Time equicontinuity]
    \label{lem:time core expanding balls}
    There exists a dense set
    $\mathcal K\subset\mathcal A$ of compactly supported finite-energy
    data with the following property. For each $\eta\in\mathcal K$ there
    are $R_\eta>0$ and $C_\eta>0$ such that
    \begin{equation}\label{eq:expanding ball core approximation}
       \|S_t^R\eta-S_s^R\eta\|_{L^1(B_R)}
       \leq C_\eta|t-s|
    \end{equation}
    for $R>R_\eta$ and all $s,t\geq0$.
\end{lemma}

\begin{proof} 
We divide the proof in several steps.

\begin{enumeratesteps}

\step[Construction of a suitable dense set of initial data $\mathcal{K} \subset \mathcal{A}$]
Set
\[
    P(s)=\frac{\Phi(s)-\Phi(0)}{\Phi(\alpha)-\Phi(0)},
    \qquad 0\leq s\leq\alpha.
\]
Since $\Phi$ is continuous and strictly increasing, $P$ is a
homeomorphism from $[0,\alpha]$ onto $[0,1]$.
Let $\mathcal W$ be the family of functions
$w\in C_c^\infty(\Rd)$ satisfying
\[
    0\leq w<1,
    \qquad
    M_w\coloneqq
    \operatorname*{ess\,sup}_{0<r<1}
    \mathcal H^{d-1}(\{w=r\})<+\infty,
\]
and define
\begin{equation*}
    \mathcal K=\{P^{-1}(w):w\in\mathcal W\}.
\end{equation*}
Thus the normalised diffusion profile $P(\eta)$ is smooth for every
$\eta\in\mathcal K$. 
Let us now first verify that $\mathcal K$ is dense in $\mathcal A$.
Given $\rho\in\mathcal A$, fix $L>0$, $0<\delta<\alpha$, and set
\[
    \widetilde\rho
    =\min\{\rho,\alpha-\delta\}\chi_{B_L}.
\]
Then $P(\widetilde\rho)$ is compactly supported and takes values
in $[0,P(\alpha-\delta)]$, with $P(\alpha-\delta)<1$. Afterwards, we approximate $P(\widetilde\rho)$ in measure by smooth functions
$w_j$, supported in a common bounded set and satisfying
\[
    0\leq w_j\leq P(\alpha-\delta)<1, \quad M_{w_j} < + \infty \quad \text{and} \quad w_j \longrightarrow P(\widetilde\rho)
    \text{ in measure}.
\]
Since $P^{-1}$ is uniformly continuous, we also have that
\[
    P^{-1}(w_j)\longrightarrow\widetilde\rho
    \quad\text{in measure}.
\]
These densities are bounded by $\alpha$ and supported in a common
bounded set, so the convergence also holds in $L^1(\Rd)$.
Thus, $\widetilde\rho$ belongs to the $L^1$-closure of $\mathcal K$.
Letting $L\to+\infty$ and $\delta\downarrow0$ proves that
$\mathcal K$ is dense in $\mathcal A$.

Every $\eta\in\mathcal K$ is compactly supported because
$P^{-1}(0)=0$. It also has finite energy, since $U(0)=0$,
$U$ is continuous on $[0,\alpha]$, and $V$ is locally bounded.





\step[Regularisation to get bounds]
Fix $w\in\mathcal W$ and write $\eta=P^{-1}(w)$.
For $0<\ee\leq1$, define
\begin{equation*}
    P_\ee(s)=
    \frac{\Phi_\ee(s)-\Phi_\ee(0)}
    {\Phi_\ee(\alpha)-\Phi_\ee(0)}.
\end{equation*}
The uniform convergence of $\Phi_\ee$ to $\Phi$ in
\eqref{eq:regularisation convergence} and strict monotonicity imply that
\begin{equation*}
\begin{aligned}
    P_\ee&\longrightarrow P
    &&\text{uniformly on }[0,\alpha],\\
    P_\ee^{-1}&\longrightarrow P^{-1}
    &&\text{uniformly on }[0,1].
\end{aligned}
\end{equation*}
Then, let us take 
\begin{equation*}
    \eta_\ee=P_\ee^{-1}(w).
\end{equation*}
All these functions have the same compact support
$K_w=\supp w$, and are such that
\begin{equation}\label{eq:regularised pressure core data convergence}
    \eta_\ee\longrightarrow\eta
    \quad\text{in }L^1(\Rd).
\end{equation}
This choice keeps the normalised diffusion profile equal to $w$.
By definition of $\eta_\ee$, it is such that
\begin{equation}\label{eq:pressure core diffusion identity}
\begin{aligned}
    \Phi_\ee(\eta_\ee)
    &=\Phi_\ee(0)
      +\bigl(\Phi_\ee(\alpha)-\Phi_\ee(0)\bigr)w,\\
    \Delta\Phi_\ee(\eta_\ee)
    &=\bigl(\Phi_\ee(\alpha)-\Phi_\ee(0)\bigr)\Delta w.
\end{aligned}
\end{equation}
Consequently, the diffusion term is uniformly bounded in $L^1(\Rd)$.
Regarding the drift term, the estimate on $\mob_\ee$ \eqref{eq:uniform regularised mobility variation}
gives
\begin{equation*}
    \int_0^\alpha|\mob_\ee'(s)|\,\diff s\leq C,
    \qquad 0<\ee\leq1.
\end{equation*}
Applying the coarea formula to $w$ and then changing variables
from $r$ to $s=P_\ee^{-1}(r)$, we obtain
\begin{align}
    \|\nabla(\mob_\ee(\eta_\ee))\|_{L^1(\Rd)}
    &=\int_0^\alpha |\mob_\ee'(s)|
      \mathcal H^{d-1}(\{w=P_\ee(s)\})\,\diff s \leq M_w\int_0^\alpha|\mob_\ee'(s)|\,\diff s
    \leq C_w.
    \label{eq:pressure mobility coarea estimate}
\end{align}
Moreover, $\mob_\ee(\eta_\ee)$ is supported in $K_w$, since
$\mob_\ee(0)=0$.
Let us now take into account 
\[
    F_\ee\coloneqq
    \Delta\Phi_\ee(\eta_\ee)
    +\diver\bigl(\mob_\ee(\eta_\ee)\nabla V\bigr),
\]
From \eqref{eq:pressure core diffusion identity} and \eqref{eq:pressure mobility coarea estimate} we get that 
\begin{align}
    \|F_\ee\|_{L^1(\Rd)}
    &\leq
    |\Phi_\ee(\alpha)-\Phi_\ee(0)|
    \|\Delta w\|_{L^1(\Rd)}\notag\\
    &\quad+
    \|\nabla V\|_{L^\infty(K_w)}
    \|\nabla(\mob_\ee(\eta_\ee))\|_{L^1(\Rd)}\notag\\
    &\quad+
    \|\Delta V\|_{L^\infty(K_w)}
    |K_w|\|\mob_\ee\|_{L^\infty(0,\alpha)}
    \leq C_\eta.
    \label{eq:pressure core uniform generator bound}
\end{align}
The constant is independent of $\ee$ and $R$: the diffusion and
mobility bounds are uniform in $\ee$, and the norms of $V$ are
taken only on the fixed compact set $K_w$.

\step[Time equicontinuity]
Choose $R_\eta>0$ such that $K_w\Subset B_{R_\eta}$, and fix
$R>R_\eta$.
Standard initial-time regularity for compatible data, as in the
bounded-domain theory used in
\cite[Section~3.1]{Carrillo_FJ_Gomez-Castro25}, gives
\begin{equation}\label{eq:pressure core generator derivative}
    \lim_{h\downarrow0}
    \frac{S_h^{R,\ee}\eta_\ee-\eta_\ee}{h} = \Delta\Phi_\ee(\eta_\ee)
    +\diver\bigl(\mob_\ee(\eta_\ee)\nabla V\bigr)
    =F_\ee
    \quad\text{in }L^1(B_R).
\end{equation} 
We now use contraction to propagate the initial derivative bound.
Fix $r>0$ and $N\in\mathbb N$, and set $h=r/N$.
The telescoping identity gives
\[
    S_r^{R,\ee}\eta_\ee-\eta_\ee
    =
    \sum_{j=0}^{N-1}
    \bigl(
        S_{(j+1)h}^{R,\ee}\eta_\ee
        -S_{jh}^{R,\ee}\eta_\ee
    \bigr).
\]
For each $j$, the semigroup property and the $L^1$-contraction 
\eqref{eq:regularised bounded contraction} imply
\begin{align*}
    \|S_{(j+1)h}^{R,\ee}\eta_\ee
      -S_{jh}^{R,\ee}\eta_\ee\|_{L^1(B_R)}
    &=
    \|S_{jh}^{R,\ee}(S_h^{R,\ee}\eta_\ee)
      -S_{jh}^{R,\ee}\eta_\ee\|_{L^1(B_R)} \leq
    \|S_h^{R,\ee}\eta_\ee-\eta_\ee\|_{L^1(B_R)}.
\end{align*}
Thus every time increment is bounded by the first one.
Taking norms in the telescoping identity and using the triangle
inequality therefore yields
\begin{align*}
    \|S_r^{R,\ee}\eta_\ee-\eta_\ee\|_{L^1(B_R)}
    &\leq
    N\|S_h^{R,\ee}\eta_\ee-\eta_\ee\|_{L^1(B_R)} =
    r\left\|
        \frac{S_h^{R,\ee}\eta_\ee-\eta_\ee}{h}
    \right\|_{L^1(B_R)}.
\end{align*}
Letting  $h\downarrow0$, and applying
\eqref{eq:pressure core generator derivative} and
\eqref{eq:pressure core uniform generator bound}, we obtain
\[
    \|S_r^{R,\ee}\eta_\ee-\eta_\ee\|_{L^1(B_R)}
    \leq r\|F_\ee\|_{L^1(B_R)}
    \leq C_\eta r.
\]
Hence, for $t\geq s$, this implies that
\begin{equation}\label{eq:regularised pressure core time modulus}
    \|S_t^{R,\ee}\eta_\ee-S_s^{R,\ee}\eta_\ee\|_{L^1(B_R)}
    \leq
    \|S_{t-s}^{R,\ee}\eta_\ee-\eta_\ee\|_{L^1(B_R)} \leq C_\eta|t-s|,
\end{equation}
with $C_\eta$ independent of $R>R_\eta$ and $\ee$. 
In order to finalise the proof, we take the sequence $\ee_k\to0$ from
Lemma \ref{lem:epsilon limit bounded ball} for this fixed $R$.
Contraction and
\eqref{eq:regularised pressure core data convergence} give
\begin{align*}
    &\sup_{t\in[0,T]}
    \|S_t^{R,\ee_k}\eta_{\ee_k}-S_t^R\eta\|_{L^1(B_R)} \\
    & \quad\leq \sup_{t\in[0,T]}
    \|S_t^{R,\ee_k}\eta_{\ee_k} - S_t^{R,\ee_k}\eta   \|_{L^1(B_R)}
    + \sup_{t\in[0,T]}\| S_t^{R,\ee_k}\eta-S_t^R\eta\|_{L^1(B_R)}\\
    &\quad\leq
    \|\eta_{\ee_k}-\eta\|_{L^1(B_R)}
    +\sup_{t\in[0,T]}
    \|S_t^{R,\ee_k}\eta-S_t^R\eta\|_{L^1(B_R)}
    \longrightarrow0.
\end{align*}
Passing to the limit in
\eqref{eq:regularised pressure core time modulus} proves
\eqref{eq:expanding ball core approximation} for $s,t\in[0,T]$.
Since $T>0$ is arbitrary, the estimate holds for all $s,t\geq0$,
with the same constant $C_\eta$ for every $R>R_\eta$.

\end{enumeratesteps}

\end{proof}

We are now ready to show global convergence as $R \to \infty$. First, we construct solutions and then we prove that the constructed solutions admit a \goodSemigroup{}.

\begin{lemma}[Global convergence]
    \label{lem:R limit global solutions energy}
    Assume \ref{hyp:mobility}--\ref{hyp:Phi} and
    \ref{hyp:free energy bdd below}--\ref{hyp:growth of V}. The sequence in
    Corollary \ref{cor:R limit local solutions} can be chosen so that, for
    every $\rho_0\in\mathcal A$ and $T>0$,
    \begin{align}\label{eq:R global L1 convergence}
       \mathcal S^{R_k,\ast}_\bullet\rho_0
       \longrightarrow S_\bullet\rho_0
       & \quad\text{in }C([0,T];L^1(\Rd)),
    \\ \label{eq:R global H-1 convergence}
       \mathcal S^{R_k,\ast}_\bullet\rho_0
       \longrightarrow S_\bullet\rho_0
       & \quad\text{in }C([0,T];H^{-1}(\Rd)).
    \end{align}
    Moreover,
    \begin{equation}\label{eq:mass final semigroup}
       \|S_t\rho_0\|_{L^1(\Rd)}=\|\rho_0\|_{L^1(\Rd)},
       \qquad 0\leq S_t\rho_0\leq\alpha.
    \end{equation}
\end{lemma}

\begin{proof} We divide the proof in several steps. 
\begin{enumeratesteps}
\step[Uniform-in-time local convergence]
Take $\mathcal{K} \subset \mathcal A$ the constructed dense set in Lemma \ref{lem:time core expanding balls}.
Fix $\eta\in\mathcal K$, $\eta$ is supported in $B_R$ for
all sufficiently large $R$, \eqref{eq:expanding ball core approximation} gives
a time modulus for $\mathcal S^{R,\ast}\eta$ which is independent of $R$:
\begin{equation*}
 \|\mathcal S_t^{R,\ast}\eta-\mathcal S_s^{R,\ast}\eta\|_{L^1(\Rd)}
 \leq C_\eta|t-s|.
\end{equation*}
Combining this equicontinuity in time with \eqref{eq:R limit local H-1} and
convergence \eqref{eq:R limit local Lp} imply
\begin{equation*}
  \mathcal S^{R_k,\ast}_\bullet\eta
  \longrightarrow S_\bullet\eta
  \quad\text{in }C([0,T];L^1(B_L))
\end{equation*}
for every $L>0$.

For arbitrary $\rho_0\in\mathcal A$, choose $\eta\in\mathcal K$. The
triangle inequality and the contraction estimates \eqref{eq:contraction R degenerate} and
\eqref{eq:main local contraction} give
\begin{align*}
 \sup_{t\in[0,T]}
 \|\mathcal S_t^{R_k,\ast}\rho_0-S_t\rho_0\|_{L^1(B_L)}
 \leq{}&2\|\rho_0-\eta\|_{L^1(\Rd)} +\sup_{t\in[0,T]}
 \|\mathcal S_t^{R_k,\ast}\eta-S_t\eta\|_{L^1(B_L)}.
\end{align*}
Hence, letting first $k\to\infty$ and then using the density of $\mathcal K$
proves
\begin{equation}\label{eq:R local continuous convergence}
  \mathcal S^{R_k,\ast}_\bullet\rho_0
  \longrightarrow S_\bullet\rho_0
  \quad\text{in }C([0,T];L^1_{\mathrm{loc}}(\Rd)).
\end{equation}

\step[Global convergence]
Set
\[
 a_L:=\sup_k\sup_{t\in[0,T]}
 \int_{\Rd\setminus B_L}\mathcal S_t^{R_k,\ast}\rho_0\,\dx.
\]
If $R_k\leq L$, the corresponding integral vanishes because
$\mathcal S_t^{R_k,\ast}\rho_0$ is extended by zero outside $B_{R_k}$; if
$R_k>L$, it is bounded by the quantity in
\eqref{eq:uniform mass tails R}. Consequently, $a_L\to0$ as
$L\to+\infty$. Therefore, up to a subsequence, 
Fatou's lemma implies that
\[
 \int_{\Rd\setminus B_L}S_t\rho_0\,\dx
 \leq\liminf_{k\to\infty}
 \int_{\Rd\setminus B_L}\mathcal S_t^{R_k,\ast}\rho_0\,\dx
 \leq a_L.
\]
The last bound is independent of $t$. Thereby
\begin{equation}\label{eq:R limit mass tails}
 \lim_{L\to+\infty}\sup_k\sup_{t\in[0,T]}
 \left(
  \int_{\Rd\setminus B_L}\mathcal S_t^{R_k,\ast}\rho_0\,\dx
  +\int_{\Rd\setminus B_L}S_t\rho_0\,\dx
 \right)=0.
\end{equation}
The combination of 
\eqref{eq:R local continuous convergence} and
\eqref{eq:R limit mass tails}, proves
\eqref{eq:R global L1 convergence}. Since both densities take values in
$[0,\alpha]$, for every $t\in[0,T]$,
\begin{align*}
 \|\mathcal S_t^{R_k,\ast}\rho_0-S_t\rho_0\|_{H^{-1}(\Rd)}
 &\leq
 \|\mathcal S_t^{R_k,\ast}\rho_0-S_t\rho_0\|_{L^2(\Rd)} \leq \alpha^{1/2}
 \|\mathcal S_t^{R_k,\ast}\rho_0-S_t\rho_0\|_{L^1(\Rd)}^{1/2}.
\end{align*}
Taking the supremum in time proves
\eqref{eq:R global H-1 convergence}. Since
\[
 \|\mathcal S_t^{R_k,\ast}\rho_0\|_{L^1(\Rd)}
 =\|\rho_0^{R_k}\|_{L^1(B_{R_k})},
\]
the global convergence \eqref{eq:R global L1 convergence} gives \eqref{eq:mass final semigroup}.

\end{enumeratesteps}
\end{proof}

Finally, we show that the obtained limit is a \goodSemigroup{} in the sense of Definition \ref{def:Free-energy dissipating semigroup}.

\begin{lemma}[Free-energy dissipating semigroup]
    \label{prop:successive limits}
    Under assumptions \ref{hyp:mobility}--\ref{hyp:Phi} and
    \ref{hyp:free energy bdd below}--\ref{hyp:growth of V}, the family
    $S=(S_t)_{t\geq0}$ obtained in Lemma
    \ref{lem:R limit global solutions energy} is a \goodSemigroup{} for
    \eqref{eq:saturation}.
\end{lemma}

\begin{proof}
We divide the proof into several steps
\begin{enumeratesteps}
\step[Semigroup properties]
Passing \eqref{eq:contraction R degenerate} through
\eqref{eq:R global L1 convergence} gives
\begin{equation}\label{eq:final contraction}
  \|S_t\rho_0-S_t\eta_0\|_{L^1(\Rd)}
  \leq\|\rho_0-\eta_0\|_{L^1(\Rd)}.
\end{equation}
For $t,h\geq0$, the triangle inequality gives
\begin{align*}
 \|S_{t+h}\rho_0-S_tS_h\rho_0\|_{L^1(\Rd)}
 \leq{}&\|S_{t+h}\rho_0-\mathcal S_{t+h}^{R_k,\ast}\rho_0\|_{L^1(\Rd)}\\
 &+\|\mathcal S_h^{R_k,\ast}\rho_0-S_h\rho_0\|_{L^1(\Rd)}\\
 &+\|\mathcal S_t^{R_k,\ast}(S_h\rho_0)-S_t(S_h\rho_0)\|_{L^1(\Rd)}.
\end{align*}
Every term tends to zero by
\eqref{eq:R global L1 convergence}; hence
\begin{equation}\label{eq:final semigroup identity}
   S_{t+h}\rho_0=S_tS_h\rho_0.
\end{equation}
The convergence in \eqref{eq:R global L1 convergence} is uniform on
$[0,T]$, and
$\mathcal S_0^{R_k,\ast}\rho_0=\rho_0\chi_{B_{R_k}}\to\rho_0$ in $L^1(\Rd)$.
It follows that
\begin{equation}\label{eq:final strong continuity}
   S_0\rho_0=\rho_0,
   \qquad
   \lim_{t\to0^+}\|S_t\rho_0-\rho_0\|_{L^1(\Rd)}=0.
\end{equation}

\step[Weak formulation and free-energy dissipation]
The weak formulation and invariant-region estimate have already been
obtained in Corollary \ref{cor:R limit local solutions}. Let
$\rho_0\in\mathcal A_{\mathcal F}$ and write $\rho_t=S_t\rho_0$.
We first pass the dissipation inequality from time zero. For the ball
solutions, put
\[
 J_k=\nabla\Phi(\rho^{R_k})+\mob(\rho^{R_k})\nabla V,
 \qquad J=\nabla\Phi(\rho)+\mob(\rho)\nabla V.
\]
The local estimates and strong convergence identify $J$ as the local
weak $L^2$ limit of $J_k$. Let us recall the free energy dissipation $\mathfrak{D}$ given by \eqref{eq:dissipation integrand convention}. Then, for each $t,L>0$, Lemma
\ref{lem:dissipation lower semicontinuity}, up to
a subsequence, gives
\begin{equation*}
 \int_0^t\int_{B_L}\mathfrak D(\rho_\sigma,J_\sigma)\,\dx\,\dsigma
 \leq\liminf_{k\to\infty}
 \int_0^t\int_{B_{R_k}}
       \mathfrak D(\rho_\sigma^{R_k},J_k)\,\dx\,\dsigma.
\end{equation*}
At the initial time, absolute integrability and $U(0)=0$ give
$\mathcal F_{R_k}[\rho_0^{R_k}]\to\mathcal F[\rho_0]$.
At the final time, \eqref{eq:R global L1 convergence} and Lemma
\ref{lem:energy expanding balls} give lower semicontinuity. Since the
energies are bounded below by $-\|g\|_1$, the lower limits of energy and
dissipation may be combined in
\eqref{eq:bounded degenerate dissipation}.
We obtain
\[
 \mathcal F[\rho_t]
 +\int_0^t\int_{B_L}
      \mathfrak D(\rho_\sigma,J_\sigma)\,\dx\,\diff\sigma
 \leq\mathcal F[\rho_0].
\]
Letting $L\to+\infty$ gives the same inequality on $\Rd$.
In particular, $\mathcal F[\rho_t]<+\infty$; the energy
convention and \eqref{eq:energy coercive lower bound} imply
that $\rho_t\in\mathcal A_{\mathcal F}$ for every $t\geq0$.

For $0\leq t_1<t_2$, apply this whole-space inequality
with initial datum $S_{t_1}\rho_0$ and time $t_2-t_1$.
By \eqref{eq:final semigroup identity}, we have
$S_s(S_{t_1}\rho_0)=\rho_{t_1+s}$.
Changing variables $t=t_1+s$ therefore gives
\begin{equation}\label{eq:energy after R limit}
 \mathcal F[\rho_{t_2}]
 +\int_{t_1}^{t_2}\int_{\Rd}
      \mathfrak D(\rho_t,J_t)\,\dx\,\dt
 \leq\mathcal F[\rho_{t_1}].
\end{equation}





\step[Time estimate]
Since
$\partial_t\rho=\diver J$ and
$|J|^2\leq\|\mob\|_{L^\infty(0,\alpha)}
\mob(\rho)|\nabla(U'(\rho)+V)|^2$, estimate
\eqref{eq:energy after R limit} gives
\begin{align*}
 \|\rho_{t_2}-\rho_{t_1}\|_{H^{-1}(\Rd)}
 \leq{}&\|\mob\|_{L^\infty(0,\alpha)}^{1/2}
 |t_2-t_1|^{1/2} \bigl(\mathcal F[\rho_{t_1}]
       -\mathcal F[\rho_{t_2}]\bigr)^{1/2},
\end{align*}
which is \eqref{eq:H-1 continuity}. Together with
\eqref{eq:final contraction}, \eqref{eq:final semigroup identity},
\eqref{eq:final strong continuity}, and \eqref{eq:mass final semigroup},
this proves the lemma.
\end{enumeratesteps}
\end{proof}

Corollary \ref{cor:R limit local solutions} proves 
\ref{item:local existence theorem}. Under the extra assumptions
\ref{hyp:free energy bdd below}--\ref{hyp:growth of V}, Lemmas
\ref{lem:tails degenerate ball semigroups}--\ref{prop:successive limits}
prove \ref{item:global existence theorem}.

\section{Long-time behaviour}\label{sec:omega limit}

In this section we prove Theorem \ref{thm:Asymptotic (P0)}. 
We study the long-time behaviour of the solutions. First, in Subsection \ref{sec:Local convergence in time}, under the standard assumptions \ref{hyp:mobility}--\ref{hyp:Phi}, we study local convergence in time. Afterwards, in Subsection \ref{sec:Global convergence in time}, adding the extra assumptions 
\ref{hyp:free energy bdd below}--\ref{hyp:growth of V} that prevent the loss of mass we upgrade the convergence to global and we describe the \timeLimit{} $S_\infty$. Finally, in Subsection \ref{sec:Energy Minimiser}, we study the energy minimisers of the free-energy functional and their role in the $\omega$-limit. 

\subsection{Local compactness of time translates}\label{sec:Local convergence in time}

First, we show local convergence through a sequence of solutions in the time intervals $[t_n, T + t_n]$, where $t_n$ is a sequence of times such that $t_n \to + \infty$. The result will follow as a consequence of Proposition \ref{lem:local nonlinear compactness}.


\begin{proof}[Proof of \ref{item:local omega compactness}]
Write $\rho^{(n)}_s=S_{t_n+s}\rho_0$. We divide the proof into several steps.


\begin{enumeratesteps}

\step[Spatial estimate
\eqref{eq:abstract nonlinear spatial bound}]
Applying \eqref{eq:Caccioppoli Phi whole space}
on $(t_n,t_n+T)$ gives
\begin{equation}\label{eq:Caccioppoli translates space}
    \sup_n\int_0^T\int_{B_{2L}}
    \chi_L^2|\nabla\Phi(\rho_s^{(n)})|^2\,\dx\,\ds
    \leq C_{L}.
\end{equation}
The constant is independent of $n$, since the Caccioppoli
estimate is uniform in the initial time of the interval.

As in the proof of Corollary
\ref{cor:R limit local solutions}, assumption
\ref{hyp:Phi} implies that $\Psi_i\circ\Phi^{-1}$
is Lipschitz. The Sobolev chain rule therefore yields
\[
    |\nabla\Psi_i(\rho^{(n)})|
    \leq C|\nabla\Phi(\rho^{(n)})|
    \quad\text{a.e. on }(0,T)\times B_L.
\]
Together with boundedness of $\Psi_i$ on $[0,\alpha]$
and \eqref{eq:Caccioppoli translates space}, this gives
\begin{equation*}
    \sup_n\sum_{i=1}^2
    \|\Psi_i(\rho^{(n)})\|_{L^2(0,T;H^1(B_L))}
    <+\infty.
\end{equation*}
Thus \eqref{eq:abstract nonlinear spatial bound} from Proposition \ref{lem:local nonlinear compactness} holds.

\step[Time estimate
\eqref{eq:abstract nonlinear time bound}]
Applying \eqref{eq:Caccioppoli time whole space}
on $(t_n,t_n+T)$ and translating the time variable gives
\begin{equation*}
    \sup_n\int_0^T
    \|\partial_s(\chi_L^2\rho_s^{(n)})\|_
    {H^{-1}(B_{2L})}^2\,\ds
    \leq C_{L}.
\end{equation*}
Again, the constant is independent of $n$.
This is \eqref{eq:abstract nonlinear time bound} from Proposition \ref{lem:local nonlinear compactness}.

\step[Application of Proposition
\ref{lem:local nonlinear compactness}]
Since $\Omega_n=\Rd$, the domain condition in Proposition
\ref{lem:local nonlinear compactness} is automatic.
The proposition therefore gives a subsequence and a limit
$\bar\rho$ satisfying
\eqref{eq:theorem time translate local H-1} and
\eqref{eq:theorem time translate local convergence}
for every $T,L>0$ and every $1 \leq p < + \infty$.

At each fixed $s\geq0$, local $H^{-1}$ convergence and
the saturation bound imply weak convergence in $L^2(B_L)$.
Consequently, $0\leq\bar\rho_s\leq\alpha$, and
\eqref{eq:main local mass inequality} gives
\[
    \int_{B_L}\bar\rho_s\,\dx
    \leq
    \liminf_{n\to+\infty}
    \int_{B_L}\rho_s^{(n)}\,\dx
    \leq \|\rho_0\|_{L^1(\Rd)}.
\]
Letting $L\to+\infty$ proves
\eqref{eq:local omega mass inequality}.
In particular, $\bar\rho\in L^1((0,T)\times\Rd)$.

\step[Weak solution] Continuity of $\Phi$ and $\mob$, strong local convergence,
and the saturation bound give
\begin{align*}
    \Phi(\rho^{(n)})
    &\longrightarrow\Phi(\bar\rho)
    &&\text{strongly in }L^2((0,T)\times B_L),\\
    \mob(\rho^{(n)})\nabla V
    &\longrightarrow\mob(\bar\rho)\nabla V
    &&\text{strongly in }L^2((0,T)\times B_L).
\end{align*}
Estimate \eqref{eq:Caccioppoli translates space},
applied with a larger cutoff if necessary, bounds the
gradients on $B_L$. The preceding strong convergence
identifies their weak limit, so
\[
    \nabla\Phi(\rho^{(n)})
    \rightharpoonup\nabla\Phi(\bar\rho)
    \quad\text{weakly in }L^2((0,T)\times B_L).
\]
In particular,
$\Phi(\bar\rho)\in L^2(0,T;H^1_{\mathrm{loc}}(\Rd))$.

Moreover, convergence in $C([0,T];H^{-1}(B_L))$ and the previous convergences give that $\bar\rho$ is a weak solution of
\eqref{eq:saturation}.


\end{enumeratesteps}

\end{proof}

\subsection{Global compactness of time translates and the time-limit operator}\label{sec:Global convergence in time}
From now on, we need to assume the extra assumptions \ref{hyp:free energy bdd below}--\ref{hyp:growth of V}. Thanks to them we recover mass-tightness and we can upgrade to global compactness. 
First, in the following Lemma we proceed to upgrade the local convergence presented in \ref{item:local omega compactness} to global convergence. 

\begin{lemma}[Global convergence]
    \label{lem:compactness time translates}
    Assume \ref{hyp:free energy bdd below}--\ref{hyp:growth of V}, and let
    $S$ be the \goodSemigroup{} presented in 
    \ref{item:global existence theorem}. Let
    $\rho_0\in\mathcal A_{\mathcal F}$ and $t_n\to+\infty$. Then, there exists 
    a subsequence, not relabelled, and $\rho_\infty\in\mathcal A$ such that,
    for every $T>0$ and $1\leq p<+\infty$,
    \begin{equation}\label{eq:time translate global convergence}
        S_{t_n+\,\bullet}\rho_0\longrightarrow\rho_\infty
        \quad\text{strongly in }L^p((0,T)\times\Rd).
    \end{equation}
    Moreover,
    $\|\rho_\infty\|_{L^1(\Rd)}=\|\rho_0\|_{L^1(\Rd)}$ and
    $S_h\rho_\infty=\rho_\infty$ for every $h\geq0$.
\end{lemma}

\begin{proof}
We consider  \ref{item:local omega compactness} as a starting point. We divide the remaining argument into two steps.

\begin{enumeratesteps}
\step[Global convergence]
By the mass preservation \eqref{eq:mass final semigroup}, the free-energy dissipation provided in Lemma \ref{prop:successive limits}, and the estimates
\eqref{eq:uniform tightness omega} and \eqref{eq:tail from V moment} apply to
the trajectory. Hence, for sufficiently large $L$,
\begin{equation}\label{eq:time translate tails}
    \sup_n\int_0^T\int_{\Rd\setminus B_L}
    \rho^{(n)}_s\,\dx\,\ds
    \leq \frac{T C(\rho_0)}{\nu_L}.
\end{equation}
Fatou's lemma gives the same bound for $\bar\rho$. Combining
\eqref{eq:time translate tails} with
\eqref{eq:theorem time translate local convergence} gives convergence in
$L^p((0,T)\times\Rd)$ for every $1 \leq p < +\infty$. It also yields
$\|\bar\rho_s\|_{L^1(\Rd)}=\|\rho_0\|_{L^1(\Rd)}$ for almost every $s$.

\step[Stationarity of the time limit]
By \eqref{eq:energy coercive lower bound}, the non-increasing function
$t\mapsto\mathcal F[S_t\rho_0]$ is bounded from below and therefore has a
finite limit. The free-energy dissipation estimate
\eqref{eq:free energy dissipation} consequently gives, for every $T>0$,
\begin{equation}\label{eq:dissipation translate vanishes}
    \int_{t_n}^{t_n+T}\int_{\Rd}
    \mob(\rho_t)|\nabla(U'(\rho_t)+V)|^2\,\dx\,\dt
    \longrightarrow0.
\end{equation}
Define
\[
    J^{(n)}\longrightarrow 0
    \quad\text{in }L^2((0,T)\times\Rd).
\]
Since the mobility $\mob$ is bounded on $[0,\alpha]$, the free-energy dissipation 
\eqref{eq:dissipation translate vanishes} implies
\[
J^{(n)}\to0$ in $L^2((0,T)\times\Rd).
\] 
On the other hand,
\eqref{eq:Caccioppoli translates space} and
\eqref{eq:theorem time translate local convergence} identify its local weak limit
as $\nabla\Phi(\bar\rho)+\mob(\bar\rho)\nabla V$. Hence this vector field
vanishes. Passing to the limit in the weak formulation gives
$\partial_s\bar\rho=0$. Thus $\bar\rho_s=\rho_\infty$ for almost every $s$,
and \eqref{eq:time translate global convergence} follows.

Fix $h>0$.
Since $\rho^{(n)}_{s_n+h}=S_h\rho^{(n)}_{s_n}$, the triangle inequality gives that 
\[
    \|S_h\rho_\infty-\rho_\infty\|_{L^1(\Rd)}
    \leq\|S_h\rho_\infty-S_h\rho^{(n)}_{s_n}\|_{L^1(\Rd)}
    +\|\rho^{(n)}_{s_n+h}-\rho_\infty\|_{L^1(\Rd)}.
\]
The $L^1$ contractive property of the \goodSemigroup{} $S$ and \eqref{eq:time translate global convergence} conclude that $S_h\rho_\infty=\rho_\infty$. 
\end{enumeratesteps}
\end{proof}


Next, we present an auxiliary lemma for the existence of a \timeLimit{}. In particular we show that the time convergence holds for all $t \to +\infty$ and not just along a subsequence $t_n \to \infty$. 


\begin{lemma}[Semigroup convergence criterion]
    \label{lem:semigroup convergence criterion}
    Assume \ref{hyp:free energy bdd below}--\ref{hyp:growth of V}, and let
    $S$ be the \goodSemigroup{} described in
    \ref{item:global existence theorem}.
    For every $\rho_0\in\mathcal A_{\mathcal F}$ there exists a stationary
    state $\rho_\infty\in\mathcal A$ such that
    \begin{equation}\label{eq:full time convergence finite energy}
        S_t\rho_0\longrightarrow\rho_\infty
        \quad\text{strongly in }L^1(\Rd)\quad\text{as }t\to+\infty.
    \end{equation}
\end{lemma}

\begin{proof}
    Apply Lemma \ref{lem:compactness time translates} to any sequence
    $t_n\to+\infty$.  From
    \eqref{eq:time translate global convergence} with $T=1$, choose
    $s_n\in(0,1)$ such that
    $S_{t_n+s_n}\rho_0\to\rho_\infty$ in $L^1(\Rd)$.  Since
    $S_h\rho_\infty=\rho_\infty$, the function
    \[
        t\longmapsto\|S_t\rho_0-\rho_\infty\|_{L^1(\Rd)} = \|S_t\rho_0-S_t\rho_\infty\|_{L^1(\Rd)} \leq \|\rho_0-\rho_\infty\|_{L^1(\Rd)}
    \]
    is non-increasing.  Its value therefore tends to zero along the full
    time variable, which proves \eqref{eq:full time convergence finite energy}.
\end{proof}

We can now show the main result of this subsection. 

\begin{proof}[Proof of \ref{item:Time limit operator}]
    The class $\mathcal A_{\mathcal F}$ is dense in $\mathcal A$ in
    $L^1(\Rd)$, as follows from spatial truncation and the finite-energy
    approximation used in the construction of the semigroup.  Let
    $\rho_0\in\mathcal A$ and choose
    $\rho_0^k\in\mathcal A_{\mathcal F}$ with
    $\rho_0^k\to\rho_0$ in $L^1(\Rd)$.  By Lemma
    \ref{lem:semigroup convergence criterion}, the limits
    \[
        \rho_\infty^k=\lim_{t\to+\infty}S_t\rho_0^k
    \]
    exist.  Passing to the limit as $t\to+\infty$ in the contraction
    estimate gives
    \[
        \|\rho_\infty^k-\rho_\infty^j\|_{L^1(\Rd)}
        \leq\|\rho_0^k-\rho_0^j\|_{L^1(\Rd)}.
    \]
    Hence $\rho_\infty^k$ converges in $L^1(\Rd)$ to a limit denoted by
    $S_\infty\rho_0$.  For every $k$ and every $t\geq0$,
    \begin{align*}
        \|S_t\rho_0-S_\infty\rho_0\|_{L^1(\Rd)}
        &\leq \|\rho_0-\rho_0^k\|_{L^1(\Rd)}
        +\|S_t\rho_0^k-\rho_\infty^k\|_{L^1(\Rd)} +\|\rho_\infty^k-S_\infty\rho_0\|_{L^1(\Rd)}.
    \end{align*}
    Letting first $t\to+\infty$ and then $k\to+\infty$ proves convergence
    for every datum in $\mathcal A$.

    The contraction property of $S_\infty$ follows by passing to the limit in
    the contraction estimate for $S_t$.  Moreover, for $h\geq0$,
    \[
        S_hS_\infty\rho_0
        =\lim_{t\to+\infty}S_hS_t\rho_0
        =\lim_{t\to+\infty}S_{t+h}\rho_0
        =S_\infty\rho_0,
    \]
    where the first equality follows from the $L^1$-continuity of $S_h$.
    The contraction property and this stationarity give, for every $T>0$,
    \[
        \sup_{0\leq s\leq T}
        \|S_{t+s}\rho_0-S_\infty\rho_0\|_{L^1(\Rd)}
        \leq\|S_t\rho_0-S_\infty\rho_0\|_{L^1(\Rd)}
        \longrightarrow0.
    \]
    This convergence can be upgraded to any $L^p$ space for $1 \leq p < \infty$ thanks to the preservation of mass and the $L^\infty$ bound of $S_t\rho_0$. Moreover,
    \eqref{eq:main mass conservation} and the global $L^1$ convergence imply
    \[
        \|S_\infty\rho_0\|_{L^1(\Rd)}
        =\|\rho_0\|_{L^1(\Rd)},
    \]
    so no mass is lost in the limit.
    Since the trajectory issued from $S_\infty\rho_0$ is constant and every
    semigroup trajectory is a weak solution, $S_\infty\rho_0$ is a
    constant-in-time weak solution.  Thus $S_\infty$ satisfies all the
    requirements in Definition \ref{def:timeLimit}.
\end{proof}

As we point out in the next remark, one of the main advantages of the confinement provided by \ref{hyp:free energy bdd below}--\ref{hyp:growth of V} is that the mass of the evolution system is preserved, since the strength of the drift potential $V$ is enough to prevent the mass from \textit{escaping to infinity}.

\begin{remark}[No mass escaping]


The global $L^1(\Rd)$ convergence obtained above already implies
preservation of mass by the time-limit operator. We record the corresponding
tail estimate, which makes the role of confinement explicit. Let
$\mathcal B\subset\mathcal A_{\mathcal F}$ be a family with uniformly
bounded masses and free energies. By
\eqref{eq:energy coercive lower bound}, there exists $C_{\mathcal B}>0$
such that
\begin{equation}\label{eq:time limit uniform moment}
    \sup_{\rho_0\in\mathcal B}\sup_{t\geq0}
    \int_{\Rd}V S_t\rho_0\,\dx
    \leq C_{\mathcal B}.
\end{equation}
Set $\nu_R=\inf_{\Rd\setminus B_R}V$, which is positive for sufficiently large $R$. Combining
\eqref{eq:time limit uniform moment} with
\eqref{eq:coercive potential} gives
\begin{equation}\label{eq:time limit mass tail}
    \sup_{\rho_0\in\mathcal B}\sup_{t\geq0}
    \int_{\Rd\setminus B_R}S_t\rho_0\,\dx
    \leq \frac{C_{\mathcal B}}{\nu_R}
    \longrightarrow0
    \qquad\text{as }R\to+\infty.
\end{equation}
The convergence $S_t\rho_0\to S_\infty\rho_0$ in $L^1(\Rd)$ and
\eqref{eq:time limit mass tail} yield the same tail bound for
$S_\infty\rho_0$. Consequently, using
\eqref{eq:main mass conservation},
\[
    \|\rho_0\|_{L^1(\Rd)}
      -\int_{B_R}S_\infty\rho_0\,\dx
    =\int_{\Rd\setminus B_R}S_\infty\rho_0\,\dx
    \leq\frac{C_{\mathcal B}}{\nu_R}
    \longrightarrow0.
\]
Thus the time limit has the same mass as the initial datum; in particular,
probability densities remain probability densities. For arbitrary
$\rho_0\in\mathcal A$, the same conclusion follows from the finite-energy
approximation and contraction argument in the proof of \ref{item:Time limit operator}.
\end{remark}

\subsection{The energy minimiser}\label{sec:Energy Minimiser}

We study the properties of the minimiser, i.e. \ref{item:Semi-group local minimiser} and \ref{item:minimiser free energy}. First, we show that the $L^1$-local minimiser is stationary with respect to the semigroup. 

\begin{proof}[Proof of \ref{item:Semi-group local minimiser}]
Let $\widehat\rho\in\mathcal A_{\mathcal F}\cap\mathcal A_M$ be an
$L^1$-local minimiser of $\mathcal F$ over $\mathcal A_M$. There exists
$\delta>0$ such that
\begin{equation}\label{eq:local minimiser neighbourhood}
    \mathcal F[\widehat\rho]\leq\mathcal F[\rho]
    \quad\text{whenever }\rho\in\mathcal A_M
    \text{ and }
    \|\rho-\widehat\rho\|_{L^1(\Rd)}<\delta.
\end{equation}
By the strong continuity of the semigroup at $t=0$, there exists
$\tau>0$ such that
\[
    \|S_t\widehat\rho-\widehat\rho\|_{L^1(\Rd)}<\delta,
    \qquad 0\leq t\leq\tau.
\]
Mass conservation gives $S_t\widehat\rho\in\mathcal A_M$. Therefore,
\eqref{eq:local minimiser neighbourhood} and the free-energy dissipation
inequality imply
\[
    \mathcal F[\widehat\rho]
    \leq\mathcal F[S_t\widehat\rho]
    \leq\mathcal F[\widehat\rho],
    \qquad 0\leq t\leq\tau.
\]
Writing $\rho_t=S_t\widehat\rho$, we consequently obtain
\begin{equation*}
    \int_0^\tau\int_\Rd
       \mob(\rho_t)
       \left|\nabla\bigl(U'(\rho_t)+V\bigr)\right|^2
       \,\dx\,\dt=0.
\end{equation*}
The associated flux
\[
    J_t=\mob(\rho_t)\nabla\bigl(U'(\rho_t)+V\bigr)
\]
therefore vanishes almost everywhere on $(0,\tau)\times\Rd$. Indeed,
since $\mob$ is bounded on $[0,\alpha]$,
\[
    |J_t|^2
    \leq \|\mob\|_{L^\infty(0,\alpha)}\,
       \mob(\rho_t)
       \left|\nabla\bigl(U'(\rho_t)+V\bigr)\right|^2.
\]
The weak formulation then yields $\partial_t\rho=0$ in
$\mathcal D'((0,\tau)\times\Rd)$.
Since the trajectory is continuous in $L^1(\Rd)$ and
$S_0\widehat\rho=\widehat\rho$, we conclude that
\[
    S_t\widehat\rho=\widehat\rho,
    \qquad 0\leq t\leq\tau.
\]
Finally, the semigroup identity extends this equality to every $t\geq0$.
This proves item \ref{item:Semi-group local minimiser}.
\end{proof}

Under the local strict convexity of the diffusion potential $U$, we show that $\widehat{\rho}$ is the unique $L^1$-local and global minimiser of the free energy over the set of initial data $\mathcal{A}$ with fixed mass $M$.

\begin{proof}[Proof of \ref{item:minimiser free energy}] We divide the proof in several steps. 

\begin{enumeratesteps}
\step[Construction of the candidate]
Put $\Upsilon (s) = \min \{\alpha ,(U')^{-1}(s)\}_+$, where we stress $\Upsilon (z)=0$ for
$z \leq U'(0^+)$ and $\Upsilon (z)=\alpha$ for $z\geq U'(\alpha^-)$.
The endpoint derivatives are understood as extended one-sided limits.
Strict convexity makes $\Upsilon $ continuous and nondecreasing. Consider
$\rho_C(x)=\Upsilon (C-V(x))$. Then $\Upsilon (z)$ is the unique minimiser of $s\mapsto U(s)-zs$ over $[0,\alpha]$. Comparison
with $s=0$ therefore gives
\[
 U(\rho_C)+V\rho_C\leq C\rho_C.
\]
Together with \eqref{eq:negative internal energy confinement}, this
implies
\begin{equation}\label{eq:candidate confinement bound}
 \bigl((1-\vartheta)V-C\bigr)\rho_C\leq g.
\end{equation}
Fix $C_0\in\mathbb R$. For all $C\leq C_0$, outside the bounded set
$\{V\leq(\max\{C_0,0\}+1)/(1-\vartheta)\}$, the coefficient on the
left of \eqref{eq:candidate confinement bound} is at least one.
Thus $\rho_C\leq g$ there, while $\rho_C\leq\alpha$ inside this set.
Hence
\[
 Q(C)=\int_{\Rd}\Upsilon (C-V(x))\,\dx
\]
is finite for every finite $C$ and is continuous by dominated
convergence. It is nondecreasing, tends to zero as $C\to-\infty$ by
dominated convergence, and tends to infinity as $C\to+\infty$ by
monotone convergence. Thus there exists a constant $C$ satisfying $Q(C)=M$.

This constant is unique. If $C_1<C_2$ gave the same finite positive
mass, monotonicity would imply $\rho_{C_1}=\rho_{C_2}$ almost everywhere
and, by continuity, everywhere. Equality at an interior density is
impossible since $U'$ is strictly increasing. The common profile would
therefore be a continuous function on the connected space $\Rd$ taking
values only in $\{0,\alpha\}$. It would be constant, contradicting its
finite positive mass.

With this unique constant, define
\begin{equation*}
 \widehat\rho(x)=\Upsilon (C-V(x)).
\end{equation*}
Equation \eqref{eq:candidate confinement bound} gives
\[
 (1-\vartheta)V\widehat\rho\leq C_+\widehat\rho+g.
\]
It follows that $V\widehat\rho\in L^1(\Rd)$. From
\eqref{eq:positive energy linear control} and
\eqref{eq:negative internal energy confinement} we get
$U(\widehat\rho)\in L^1(\Rd)$, so that
$\widehat\rho\in\mathcal A_{\mathcal F}$. 
Set $f(s)=U(s)$. The pointwise minimising property is equivalently
\begin{equation}\label{eq:minimiser pointwise variational inequality}
 f(s)+V(x)s-f(\widehat\rho(x))-V(x)\widehat\rho(x)
 \geq C\bigl(s-\widehat\rho(x)\bigr),
 \qquad 0\leq s\leq\alpha.
\end{equation}

\step[Global and local minimality]
Let $\rho\in\mathcal A_M\cap\mathcal A_{\mathcal F}$. Substituting
$s=\rho(x)$ in
\eqref{eq:minimiser pointwise variational inequality}, integrating over
$\Rd$, and using
$\|\rho\|_{L^1}=\|\widehat\rho\|_{L^1}=M$, we obtain
\begin{equation*}
    \mathcal F[\rho]-\mathcal F[\widehat\rho]
       \geq C\int_{\Rd}(\rho-\widehat\rho)\,\dx=0.
\end{equation*}
Thus $\widehat\rho$ is a global, and therefore local, minimiser. Densities
outside $\mathcal A_{\mathcal F}$ have infinite energy and
do not alter this conclusion.

The admissible set $\mathcal A_M$ is convex. Since $U$ is strictly convex, so is 
$\mathcal F$ on this set; the potential-energy term is
linear. Hence the global minimiser is unique.

The truncation and bounded mass correction in the proof of Lemma
\ref{lem:energy expanding balls} also give $L^1$ approximation of every
$\rho\in\mathcal A_M$ by densities of the same mass with bounded
support. Each approximant has finite energy, since $U$ is continuous
on $[0,\alpha]$ and $V$ is locally bounded. Thus a density of infinite
energy cannot be an
$L^1$-local minimiser. It remains to exclude any other finite-energy
$L^1$-local minimiser. Let
$\rho_*\in\mathcal A_M\cap\mathcal A_{\mathcal F}$ be an $L^1$-local
minimiser and suppose that $\rho_*\neq\widehat\rho$. For
$0<\lambda<1$, set
\[
    \rho_\lambda=(1-\lambda)\rho_*+\lambda\widehat\rho.
\]
Then $\rho_\lambda\in\mathcal A_M$ and
$\rho_\lambda\to\rho_*$ in $L^1(\Rd)$ as $\lambda\to0^+$. Since
$\widehat\rho$ is the unique global minimiser, strict convexity gives
\[
    \mathcal F[\rho_\lambda]
    <(1-\lambda)\mathcal F[\rho_*]
      +\lambda\mathcal F[\widehat\rho]
    \leq\mathcal F[\rho_*],
\]
contradicting the local minimality of $\rho_*$ for sufficiently small
$\lambda$. Thus every $L^1$-local minimiser coincides with
$\widehat\rho$. This proves item \ref{item:minimiser free energy} and
completes the proof of Theorem \ref{thm:Asymptotic (P0)}.
\end{enumeratesteps}
\end{proof}

\section{Numerical experiments}\label{sec:Numerical experiments}

\renewcommand{\topfraction}{0.9}
\renewcommand{\bottomfraction}{0.8}
\renewcommand{\textfraction}{0.05}

We give some simple numerical examples that illustrate the behaviour of problem \eqref{eq:saturation}. For these experiments, we will use the implementation from \cite{Carrillo_FJ_Gomez-Castro25} over $\R$ instead of a bounded interval with no-flux boundary conditions. This implementation corresponds to an implicit finite volume scheme introduced for this problem in \cite{BCH23}. We use the \texttt{julia} language \cite{Bezanson_Edelman_Karpinski_Shah17} and the \texttt{NLSolve.jl} package (see \cite{NLSolver}). We want to emphasise the difference of qualitative behaviours of the problem \eqref{eq:saturation} regarding whether the mass tightness assumptions \ref{hyp:free energy bdd below}--\ref{hyp:growth of V} are fulfilled or not. Depending on the choice we make the mass is \textit{kept} or \textit{lost at infinity}.

\paragraph{Quadratic potential.} We exemplify the behaviour for the nonlinear diffusion $U(\rho)=\rho^2$ with a coercive potential $V(x) = 10 x^2$. This shows how under the extra assumptions \ref{hyp:free energy bdd below}--\ref{hyp:growth of V} the mass is preserved and it produces free boundaries at levels $\rho = 0 , \alpha$, see Figure \ref{fig:convex V}.


\begin{figure}[pos=htbp]
    \centering 
    \includegraphics[width=.79\textwidth]{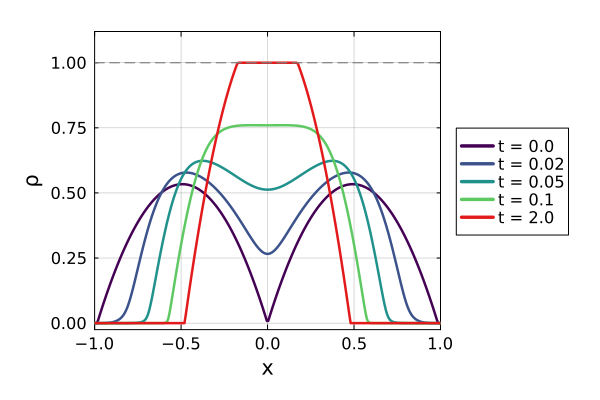}
    \caption{$\mob(\rho) = \rho(1-\rho)$, $U(\rho) = \rho^2$ and $V(x) = 10x^2$. $\Delta t = \Delta x = 2^{-7}$.}
    \label{fig:convex V}
\end{figure}

\paragraph{Partial escape of mass at infinity.} We take again the nonlinear diffusion potential $U(\rho) = \rho^2$. However, in this case the potential $V$ is not coercive anymore since it is chosen such that $\inf_{\R \backslash (-\sigma, \sigma)} V(x) \rightarrow 0$ as $\sigma \to +\infty$ (see Figure \ref{fig:Mass loss} above). Due to the lack of the mass-tightness assumptions \ref{hyp:free energy bdd below}--\ref{hyp:growth of V}, we observe how a significant part of the mass is \textit{lost at infinity}, see Figure \ref{fig:Mass loss}.

\begin{figure}[pos=htbp]
    \centering 
    \includegraphics[width=.79\textwidth]{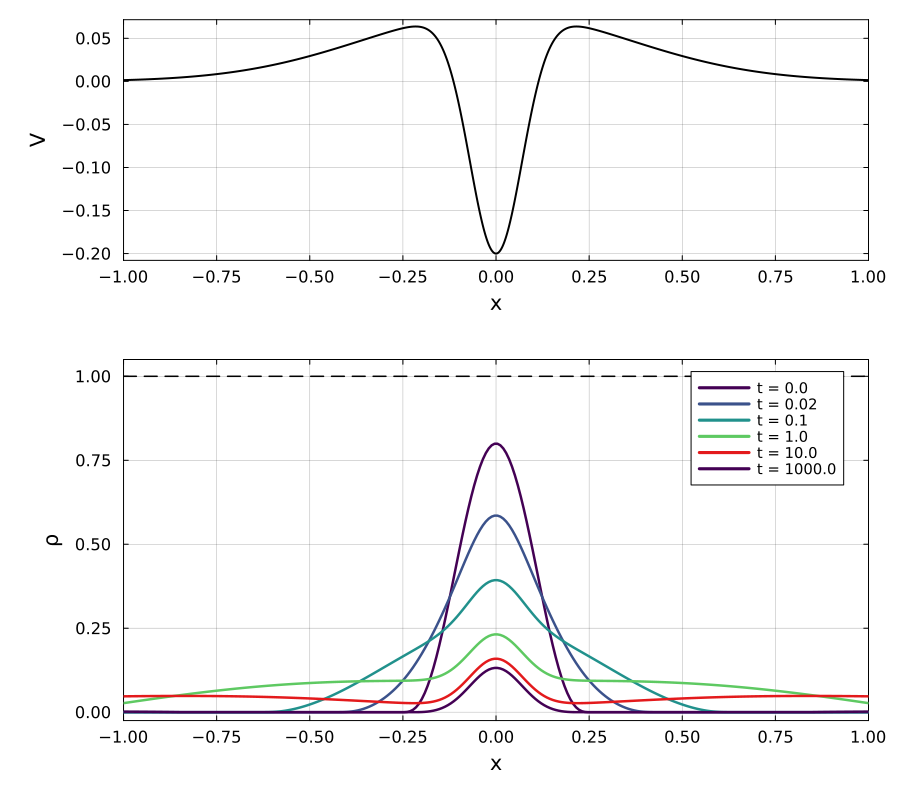}
    \caption{$\mob(\rho) = \rho(1-\rho)$, $U(\rho) = \rho^2$. $\Delta t = 0.05$, $\Delta x = 2^{-7}$.}
    \label{fig:Mass loss}
\end{figure}

\FloatBarrier

\section*{Acknowledgements}
AFJ is thankful to Bruno Volzone (Politecnico di Milano) for fruitful discussions at an early stage of the project.

\bibliographystyle{main18-nourl}

\bibliography{bibliography}


\end{document}